%% file: two-patch-final.tex
\documentclass[smallextended]{svjour3}
\smartqed

\usepackage{graphicx}
\usepackage{amssymb}
\usepackage{amsmath}
\usepackage{theorem}
\usepackage{epsfig}
\usepackage{subfigure}
\usepackage{color}
\usepackage{newlfont}
\usepackage{multirow}
\usepackage{longtable} 
\usepackage{ifthen}
\usepackage{alltt}
\usepackage{enumerate}

\newcommand{\ba}{\begin{array}}
\newcommand{\ea}{\end{array}}
\newcommand{\bae}{\begin{eqnarray}}
\newcommand{\eae}{\end{eqnarray}}
\newcommand{\bea}{\begin{eqnarray*}}
\newcommand{\eea}{\end{eqnarray*}}
\newcommand{\be}{\begin{equation}}
\newcommand{\ee}{\end{equation}}

\def\to{{\rightarrow}}

\newcommand{\U}{\mathcal U}
\newcommand{\V}{\mathcal V}

\def\R{{\mathbb{R}}}

\def\to{{\rightarrow}}

\usepackage{lscape}
\usepackage{graphicx}
\usepackage{epstopdf}
\newcommand{\ep}{\mbox{$\varepsilon$}}

\newcommand{\E}{\mathbb{E}}
\journalname{Journal of Mathematical Biology}

\begin{document}

\title{Expansion or extinction: deterministic and stochastic two-patch models with Allee effects}
\titlerunning{Deterministic and stochastic models with Allee effects}

\author{Yun Kang \and Nicolas Lanchier}

\institute{Yun Kang \at Applied Sciences and Mathematics, Arizona State University, Mesa, AZ 85212, USA. \\
          \email{yun.kang@asu.edu} \and
           Nicolas Lanchier \at School of Mathematical and Statistical Sciences, Arizona State University, P.O. Box 871804, Tempe, AZ 85287, USA. \\
          \email{lanchier@math.asu.edu}}

\date{Received: date / Accepted: date}

\maketitle

\begin{abstract}
 We investigate the impact of Allee effect and dispersal on the long-term evolution of a population in a patchy environment, focusing
 on whether a population already established in one patch either successfully invades an adjacent empty patch or undergoes a global
 in-all-patch extinction.
 Our study is based on the combination of analytical and numerical results for both a deterministic two-patch model and its stochastic
 analog.
 The deterministic model has either two or four attractors.
 In the presence of weak dispersal, the analysis of the deterministic model shows that a high-density and a low-density populations
 can coexist at equilibrium in nearby patches, whereas the analysis of the stochastic model indicates that this equilibrium is metastable,
 thus leading after a large random time to either an in-all-patch expansion or an in-all-patch extinction.
 Up to some critical dispersal, increasing the intensity of the interactions leads to an increase of both the basin of attraction of the
 in-all-patch extinction and the basin of attraction of the in-all-patch expansion.
 Above this threshold, while increasing the intensity of the dispersal, both deterministic and stochastic models predict a synchronization
 of the patches resulting in either a global expansion or a global extinction:
 for the deterministic model, two of the four attractors present when the dispersal is weak are lost, while the stochastic
 model no longer exhibits a metastable behavior.
 In the presence of strong dispersal, the limiting behavior is entirely determined by the value of the Allee threshold as the
 global population size in the deterministic and the stochastic two-patch models evolves as dictated by the single-patch counterparts.
 For all values of the dispersal parameter, Allee effects promote in-all-patch extinction in terms of an expansion of the
 basin of attraction of the extinction equilibrium for the deterministic model and an increase of the probability of extinction for
 the stochastic model.

\keywords{Allee effect \and Deterministic model \and Stochastic model \and Extinction \and Expansion \and Invasion \and
 Bistability \and Basin of attraction \and Metastability}


\end{abstract}


\section{ Introduction}
\label{sec:intro}

\indent Biological invasions of alien species are commonly divided into three stages: arrival, establishment, and expansion
 (Liebhold and Tobin 2008).
 The precise circumstances of an alien species' arrival, which refers to the transport of an alien species to new areas outside of its
 native range, are generally not known and are not the purpose of this article.
 The establishment stage refers to a growth phase of the population density up to some threshold above which it is usually assumed that natural
 extinction is highly unlikely.
 However, if during the expansion stage, which refers to the spreading of the alien species to nearby new areas, the population expands
 in space through dispersal without significantly increasing its size, thus leading to a drop of its density, there might be a risk of
 extinction for species subject to an Allee effect.
 The Allee effect refers to a certain process that leads to decreasing net population growth with decreasing density, thus inducing the
 existence of a so-called Allee threshold below which populations are driven toward extinction.
 The causes of Allee effect identified by ecologists are numerous.
 They include failure to locate mates (Hopper and Roush 1993; Berec \emph{et al} 2001), inbreeding depression (Lande 1998), failure to satiate
 predators (Gascoigne and Lipcius 2004), lack of cooperative feeding (Clark and Faeth 1997), etc.
 Stochasticity, e.g., demographic and/or environmental stochasticity, may also play an important role during the critical time period when
 an alien species already established in one area starts to spread its population into a new area through dispersal.
 In this article,
 we think of the establishment stage as a local expansion of the population in a given geographical location, which involves an increase of
  population density in this location, while we think of the expansion stage as a global expansion of the population
 in space into nearby geographical locations regardless of its density.
 We call a global expansion successful if it leads to the population being established in nearby geographical locations, and unsuccessful if
 on the contrary the population fails to get established in new locations which may also lead to a global extinction (the population goes
 extinct in all patches).
 The main purpose of this article is to study the critical time period when a species already established in a specific geographical
 location starts to expand in space, and determine whether the expansion stage is successful or not.
 Both Allee effect and stochasticity are central to better understand why some alien species successfully expand into new geographical
 areas, and there has been recently a growing recognition of the importance of these two components in biological invasions (Drake 2004;
 Leung \emph{et al} 2004; Taylor and Hastings 2005; Ackleh \emph{et al} 2007).
 Understanding their role and strength is of critical importance to gain some insight into why some species are more invasive than others,
 and may suggest some proper biological control strategies to regulate some populations (Liebhold and Tobin 2008).

\indent If an alien species subject to an Allee effect establishes its population in one area, i.e., its population is above the Allee
 threshold in this area, then the first step of population expansion is to spread to a nearby new area where the population is either
 absent or at least below the Allee threshold.
 A natural way to model this situation is to consider a two-patch model with heterogeneous initial conditions such that
\begin{enumerate}
 \item both patches are coupled by interacting through dispersal, and \vspace{4pt}
 \item in the absence of interactions, i.e., when the patches are uncoupled, the initial conditions lead to establishment
  in one patch and extinction in the other patch.
\end{enumerate}
 This approach has been used previously by Alder (1993) and Kang \emph{et al} (2009).
 In this article, we follow this modeling strategy to study the global expansion (population above the threshold in both patches) and
 global extinction (population below the threshold in both patches) of an alien species subject to an Allee effect during the critical
 time period between the establishment stage and the expansion stage by employing both a deterministic two-patch model and its
 stochastic analog.
 The objectives of our study are twofold:
 the first is to study the consequences of the inclusion of dispersal and Allee effect on the extinction and expansion
 for both deterministic and stochastic models with heterogeneous initial conditions;
 the second is to understand the effects of stochasticity by comparing the results based on both models.

\indent There is a copious amount of literature on the invasion and extinction of populations subject to Allee effects (e.g.,
 Dennis 1989, 2002; Veit and Lewis 1996;  McCarthy 1997; Shigesada and Kawasaki 1997; Greene and Stamps 2001;
 Keitt \emph{et al} 2001; Fagan \emph{et al} 2002; Wang \emph{et al} 2002; Liebhold and Bascompte 2003; Schreiber 2003;
 Zhou \emph{et al} 2004; Petrovskii \emph{et al} 2005; Taylor and Hastings 2005) which also includes various models in patchy
 environment (e.g., Amarasekare 1998a, 1998b; Gyllenberg \emph{et al} 1999; Ackleh \emph{et al} 2007; Kang \emph{et al} 2009).

\indent In the deterministic side, Amarasekare (1998a, 1998b) investigated how an interaction between local density dependence,
 dispersal, and spatial heterogeneity influence population persistence in patchy environments.
 In particular, she studied how Allee (or Allee-like) effects arise from these patchy models.
 Gyllenberg \emph{et al} (1999) studied a deterministic model of a symmetric two-patch metapopulation to determine conditions that
 allow the Allee effect to conserve and create spatial heterogeneities in population densities.
 Rather than exploring the global dynamics of their models, both Amarasekare (1998a, 1998b) and Gyllenberg \emph{et al} (1999)
 studied the influence of an Allee effect on local dynamics, e.g., number of equilibriums and local stability.
 There are few studies regarding the influence of an Allee effect on the extinction versus expansion of populations in patchy
 environments (e.g., Ackleh \emph{et al} 2007; Kang \emph{et al} 2009).
 Kang \emph{et al} (2009) studied the influence of an \emph{Allee-like effect} for a discrete-time two-patch model on
 plant-herbivore interactions where patches are coupled through a dispersal.
 Their study suggests that for a certain range of dispersal parameters the population of herbivores in both patches drops under
 the Allee threshold, thus leading to an extinction of the herbivores in both patches, for the majority of positive initial
 conditions.

\indent In the stochastic side, the recent work by Ackleh \emph{et al} (2007) focuses on a multi-patch population model
 combining stochasticity and Allee effect.
 Their numerical simulations show that populations with initial sizes below but near their Allee threshold in each patch
 can still become established and invasive if stochastic processes affect life history parameters.
 The closer the population to its Allee threshold, the greater the probability of invasion.
 A more theoretical approach based on interacting particle systems can be found in Krone (1999).
 In his model, each site of the infinite integer lattice has to be thought of as a patch which is either empty, occupied by a
 small colony with a high risk of going extinct, or occupied by a full colony with a longer life span.
 If successful, a small colony gets established to become a full colony, while empty patches get colonized by a small colony
 due to invasions from adjacent full colonies, making space explicit.

\indent In this paper, although we model the population dynamics deterministically following the approach of
 Amarasekare (1998a, 1998b), Gyllenberg \emph{et al} (1999) and Ackleh \emph{et al} (2007), our stochastic process as well as
 analytical results for both models are new.
 For the deterministic model, our focus is on the global dynamics of the system combining dispersal and Allee effects.
 In particular, we give analytical results on how Allee threshold and dispersal affect the geometry of the basins of attraction
 of the stable equilibriums.
 The stochastic model is derived from the deterministic one using a process that has two absorbing states corresponding to global
 extinction and global expansion, which allows to have a rigorous definition of successful invasion.
 In particular, our model is designed to study analytically the probability that a fully occupied patch successfully invade
 a nearby empty patch.
 To gain insight into the effects of stochasticity on the population dynamics, we will compare in details the results obtained
 for both models.

\indent The rest of the article is organized as follows.
 In section \ref{sec:deterministic}, we introduce the deterministic two-patch model with Allee effect coupled by dispersal.
 Based on the analysis of the invariant sets, we give a complete picture of the global dynamics of the system including the existence
 of the nontrivial locally stable equilibriums and the geometry of their basin of attraction.
 Numerical simulations have been performed to gain some insight into how dispersal and Allee threshold affect the exact basin of
 attraction of the equilibriums.
 In section~\ref{sec:stochastic}, we introduce and analyze mathematically the stochastic model focusing on the time to fixation of the
 process, the existence of metastable states and the probability of a successful invasion when starting from heterogeneous initial
 conditions.
 Simulations of the stochastic model have also been performed to better understand these aspects.
 In section \ref{sec:comparison}, we compare the predictions based on both models, and describe the biological implications of
 our analytical and numerical results.
 Finally, Section \ref{sec:proofs} is devoted proofs.


\section{ A simple deterministic two-patch model with Allee effects}
\label{sec:deterministic}

\indent The first step in constructing the deterministic two-patch model is to consider single-species dynamics including
 an Allee affect as potential candidates to describe the evolution in a single patch.
 The two-patch model is then naturally derived by looking at a two-dimensional system in which both components are coupled
 through dispersal.
 The ecological dynamics of a single species' population subject to an Allee effect that can mimic the dynamics in the
 absence of dispersal is usually described by the model
\begin{equation}
\label{sp}
 \dot{x} \ = \ G (x) \,x \ - \ H (x)
\end{equation}
 where $x(t)$ denotes the population density at time $t$.
 The function $G$ measures the logistic component of population growth, which is given by
\begin{equation}
\label{g}
 G (x) \ = \ r \ - \ ax
\end{equation}
 where $r$ is the per capita intrinsic growth rate and $a$ measures the extra mortality caused by intraspecific competition.
 In general, the bistability of the differential equation \eqref{sp} is triggered by combining the negative density-dependence of the
 logistic growth $G$ with the positive density-dependence of an additional demographic factor represented here by the function $H$.
 The decreasing reproduction due to a shortage of mating encountered in low population density and the decreasing mortality due
 to the weakening predation risk in higher population density are two important examples of such factors (Stephens and Sutherland 1999)
 which, following Dercole \emph{et al} (2003), can be modeled by a Holling type II functional response: $H (x) = c x / (x + d)$.
 The resulting population model creates, under suitable parameter values, a threshold below which the population goes extinct
 eventually and above which the population density approaches a positive equilibrium.
 The simplest and generic model that captures the population dynamics of a single species with Allee effects can be described by
\begin{equation}
\label{rallee_s}
 \dot{x} \ = \ r x \,(x - \theta) (1 - x)
\end{equation}
 where $r$ is the per capita intrinsic growth rate after rescaling and $\theta$ is a threshold that lies between 0 and 1 after
 rescaling.
 The latter, called Allee threshold, determines whether the population goes extinct or establishes itself.
 More precisely, the population dynamics of \eqref{rallee_s} can be summarized as follows.
\begin{lemma}[Single species dynamics with Allee effects]
\label{l_sae}
 If the population of a single species is described by \eqref{rallee_s}, then it goes extinct when $x (0) < \theta$ while its
 density goes to 1 when $x (0) > \theta$.
\end{lemma}
 Thinking of model \eqref{rallee_s} as describing the population dynamics in one patch, the dynamics of two interacting
 identical patches with dispersal $\mu$ can be modeled by
\begin{eqnarray}
\label{d_x}
 \dot{x} & = & r x \,(x - \theta) (1 - x) + \mu \,(y - x) \\
\label{d_y}
 \dot{y} & = & r y \,(y - \theta) (1 - y) + \mu \,(x - y)
\end{eqnarray}
 where $\mu \in [0, 1]$ is a dispersal parameter, representing the fraction of population migrating from one patch to another
 per unit of time.
 Although the system \eqref{d_x}-\eqref{d_y} is symmetric in $x$ and $y$, asymmetry will be introduced by considering different
 initial conditions in each patch: $x (0) \neq y (0)$.
 We will pay a particular attention to situations where one patch is initially below and the other patch above the Allee threshold,
 in which case, in the absence of dispersal, the population goes extinct in the first patch but establishes itself in the second one.
 The main objective is to understand, based on analytical and numerical results, how the dispersal parameter $\mu$ and the Allee
 threshold $\theta$ affect the global dynamics, i.e., the limit sets of the system \eqref{d_x}-\eqref{d_y} and the geometry of
 their basin of attraction.

\indent Our analytical results suggest the following picture of the global dynamics.
 Recall first that, in the absence of dispersal, the system has four locally stable equilibrium points, which correspond to cases
 when the population in each patch either goes extinct or gets established.
 In the presence of dispersal, the existence of (stable) limit cycles is also excluded: starting from almost every initial condition
 in $\R_+^2$, the system converges to an equilibrium point.
 This is partly proved analytically in Theorem \ref{th:simple} and supported by the numerical simulations of Figure \ref{fig:ODE}.
 Therefore, we focus our attention on the existence, stability and basins of attraction of the equilibrium points.
 The effects of the dispersal parameter and the value of the Allee threshold are as follows.
 First, the dynamics of the deterministic two-patch model in the presence of weak dispersal are similar to that of the uncoupled
 system, having four locally stable equilibriums (Theorem \ref{th:attractors}), i.e., the extinction state $(0, 0)$, the expansion
 state $(1, 1)$ and two asymmetric interior equilibriums $(x_s,y_s)$, $(y_s, x_s)$, which is not retained after the inclusion of
 demographic stochasticity, which induces two absorbing states and two metastable states.
 While increasing the dispersal parameter from 0, the basin of attraction of the extinction state $(0, 0)$ and expansion state
 $(1, 1)$ increase until a certain critical value at which both patches interact enough to synchronize, which drives the system
 to either global extinction $(0, 0)$ or global expansion $(1, 1)$:
 there are only two attractors (Theorems \ref{th:basin} and \ref{th:attractors}).
 Above this critical value, dispersal promotes extinction when the Allee threshold is below one half but promotes survival when
 the Allee threshold is above one half (Theorems \ref{th:basin} and \ref{th:dispersal}).
 Finally, in the presence of a strong dispersal, both patches synchronize fast enough so that the global dynamics reduce to
 that of a single-patch model: if the initial global density, i.e., the average of the densities in both patches, is below the
 Allee threshold then the population goes extinct whereas if it exceeds the Allee threshold then the population expands
 globally (Theorem \ref{th:dispersal}).
 In other respects, for any value of the dispersal parameter, increasing the Allee threshold promotes extinction, and populations
 initially below the Allee threshold in both patches are doomed to extinction, whereas populations initially above the Allee
 threshold in both patches expand globally.
 These results are stated rigorously in the following two subsections. Simulation results and detailed summary are given in the
 last subsection.


\subsection{Global dynamics and basins of attraction}

\indent In order to understand the global dynamics of the deterministic two-patch model, the first step is to identify its omega
 limit sets.
 Since the model is simply a two-dimensional ODE, its omega limit sets are either equilibrium points or limit cycles according
 to the Poincar\'e-Bendixson theorem.
 As stated in the next theorem, when the dispersal parameter is sufficiently large, an application of the Dulac's criterion
 reveals simple dynamics by excluding the existence of limit cycles: for any initial condition, the system converges to an
 equilibrium point.
\begin{theorem}[Simple dynamics]
\label{th:simple}
 For any $c \in [0, 3)$, $r > 0$ and $\theta \in (0, 1)$, if
\begin{equation}
\label{simple_values}
 \mu \ \geq  \ r \,\theta \,(c - 1) \ + \ \frac{r \,(2 - c)^2 \,(1 + \theta)^2}{4 \,(3 - c)}
\end{equation}
 then every trajectory of \eqref{d_x}-\eqref{d_y} converges to an equilibrium point.
\end{theorem}
 Theorem \ref{th:simple} indicates for instance that every trajectory of the system \eqref{d_x}-\eqref{d_y} converges
 to an equilibrium point under the condition $\mu \geq r (\theta^2 - \theta + 1) / 3$ if one takes $c = 0$.
 In addition, Theorem \ref{th:basin} below implies that if limit cycles emerge for smaller values of the dispersal parameter then
 each of them is included in one of the two regions of the phase space in which the population lies above the Allee threshold in
 one patch but below the Allee threshold in the other patch, i.e.,
 $$ \{(x, y) \in \Omega_0 : 0 \leq x \leq \theta \ \hbox{and} \ \theta \leq y \leq 1 \} \quad \hbox{and} \quad
    \{(x, y) \in \Omega_0 : 0 \leq y \leq \theta \ \hbox{and} \ \theta \leq x \leq 1 \} $$
 where $\Omega_0 = \R_+^2$.
 Numerical simulations (see Figure \ref{fig:ODE}) further suggest that, for any value of the dispersal parameter, there is no
 stable limit cycle, which implies that locally stable equilibriums are the only possible attractors of the system, so we focus
 our attention on the existence, stability and basins of attraction of the equilibrium points.
 We also would like to point out that if the system has no Allee effect, e.g., a metapopulation model coupled by both competition
 and migration with uniparental reproduction, then this system admits no periodic solutions (Proposition 1 in Gyllenberg \emph{et al} 1999).

\indent It can be easily seen that the system \eqref{d_x}-\eqref{d_y} has three symmetric equilibriums for all positive values of the
 parameters:
 one boundary equilibrium given by $E_0 = (0, 0)$ and two interior equilibriums given respectively by $E_{\theta} = (\theta, \theta)$
 and $E_1 = (1, 1)$.
 For obvious reasons, we call $E_0$ the extinction state of the system and $E_1$ the expansion state.
 Theorem \ref{th:basin} below indicate that, for all parameter values, these two trivial equilibriums are locally stable whereas the
 interior equilibrium point $E_{\theta}$ is unstable.
 Hence, to understand the global dynamics of the system, the next step is to study the geometry of the basins of attraction of the
 two trivial equilibriums, i.e.,
 $$ \begin{array}{rcl}
     B_0 \ & = & \ \{(x (0), y (0)) \in \Omega_0 : \lim_{t \to \infty} (x (t), y (t)) = E_0 \} \vspace{4pt} \\
     B_1 \ & = & \ \{(x (0), y (0)) \in \Omega_0 : \lim_{t \to \infty} (x (t), y (t)) = E_1 \}. \end{array} $$
 Letting $\Omega_{0, \theta}$ and $\Omega_{\theta}$ denote the subsets
 $$ \begin{array}{rcl}
    \Omega_{0, \theta} & = & \{(x, y) \in \Omega_0 : 0 \leq x \leq \theta \ \hbox{and} \ 0 \leq y \leq \theta \} \vspace{4pt} \\
    \Omega_{\theta}    & = & \{(x, y) \in \Omega_0 : x \geq \theta \ \hbox{and} \ y \geq \theta \} \end{array} $$
 Lemma \ref{l_sae} indicates that, in the absence of dispersal, the basins of attraction of $E_0$ and $E_1$ for the (uncoupled) system
 are given by $B_0 = \Omega_{0, \theta} \setminus E_{\theta}$ and $B_1 = \Omega_{\theta} \setminus E_{\theta}$.
 The following theorem shows how the inclusion of a dispersal affects the basins of attraction.
\begin{theorem}[Local stability and basins of attraction]
\label{th:basin}
\begin{enumerate}
 \item The extinction state $E_0$ and expansion state $E_1$ are always locally stable whereas the interior fixed point
  $E_{\theta}$ is always unstable. \vspace{4pt}
 \item If $2 \mu > r \theta (1 - \theta)$ then $E_{\theta}$ is a saddle while if $2 \mu < r \theta (1 - \theta)$ then
  $E_{\theta}$ is a source. \vspace{4pt}
 \item $\Omega_{0, \theta} \setminus E_{\theta} \subset B_0$.
  If in addition $\theta < 1/2$ then
  $$ B_0 \ \subset \ \{(x, y) \in \Omega_0 : x + y < 2 \theta \}. $$
 \item $\Omega_{\theta}\setminus E_{\theta} \subset B_1$.
  If in addition $\theta > 1/2$ then
  $$ B_1 \ \subset \ \{(x, y) \in \Omega_0 : x + y > 2 \theta \}. $$
\end{enumerate}
\end{theorem}
 Theorem \ref{th:basin} indicates that the inclusion of dispersal promotes both global extinction and global expansion of the
 system, as the basins of attraction of both equilibrium points $E_0$ and $E_1$ are larger in the presence than in the absence
 of dispersal.
 Numerical simulations further suggest that, up to a certain critical value, increasing the dispersal parameter translates
 into an increase of $B_0$ and $B_1$.
 The value of the Allee threshold $\theta$ also plays an important role in the global dynamics.
 When the Allee threshold lies below one half, which is common in nature, the largest possible basin of attraction of $E_0$ is
 $$ \{(x, y) \in \Omega_0 : x + y < 2 \theta \}. $$
 Moreover, according to numerical simulations (see Figure \ref{fig:ODE}), increasing the Allee threshold promotes extinction of
 the system in the sense that, the dispersal parameter being fixed, the smaller the Allee threshold, the smaller the basin
 attraction of the extinction state $E_0$ and the larger the basin attraction of the expansion state $E_1$.  
 Finally, we would like to point out that parts 1 and 2 of the theorem hold for the system \eqref{d_x}-\eqref{d_y} but not
 always for two-patch models with Allee effect.
 A counter example is provided by the metapopulation model coupled by both competition and migration with biparental reproduction
 studied by Gyllenberg \emph{et al} (1999).
 The first step to prove parts 3 and 4 of Theorem \ref{th:basin} will be to identify the positive invariant sets of the
 system \eqref{d_x}-\eqref{d_y} included into the upper right quadrant $\Omega_0 = \R_+^2$.
 Recall that a set is called positive invariant if any trajectory starting from this set stays in this set at all future times.
 Since we are interested in the global dynamics of the system, our objective will be to find all the possible invariant sets
 in $\Omega_0$.
 Notice that the union and the intersect of positive invariant sets are also positive invariant.
 All these positive invariant sets have an important role in understanding the dynamics in regions of the phase space where the
 population is below the Allee threshold in one patch but above the Allee threshold in the other patch.
 In particular, they will give us means of decomposing the phase space by restricting our attention to the dynamics on each
 invariant set and then sewing together a global solution from the invariant pieces.


\subsection{ Dispersal effects and multiple attractors}

\indent In this subsection, we study the effects of the dispersal parameter on the dynamics of the two-patch model when the Allee
 threshold is fixed.
 Theorem \ref{th:simple} suggests that the number of attractors is also equal to the number of locally stable equilibriums.
 Our study shows that the value of the dispersal parameter determines the number of equilibriums, thus the possible number
 of attractors.
 Let $S_{\theta}$ denote the stable manifold of the unstable interior equilibrium $E_{\theta}$, i.e.,
 $$ S_{\theta} \ = \ \{(x (0), y (0)) \in \Omega_0 : \lim_{t \to \infty} (x (t), y (t)) = E_{\theta} \}. $$
 The following theorem indicates that, when the dispersal is sufficiently large, both patches interact enough to synchronize,
 which drives the system to either global extinction or global expansion: there are only two stable equilibriums, the extinction
 state $E_0$ and the expansion state $E_1$.
\begin{theorem}[Large dispersal]
\label{th:dispersal}
 Assume that
\begin{equation}
\label{eq:dispersal}
 \mu \ > \ \frac{r (\theta^2 - \theta + 1)}{6}.
\end{equation}
 Then, the system \eqref{d_x}-\eqref{d_y} has only two attractors: $E_0$ and $E_1$. Moreover,
\begin{enumerate}
\item If $\theta < 1/2$ and \eqref{eq:dispersal} holds then
 $$ \{(x, y) \in \Omega_0 \setminus S_{\theta} : x + y \geq 2 \theta \} \ \subset \ B_1. $$
\item If $\theta > 1/2$ and \eqref{eq:dispersal} holds then
 $$ \{(x, y) \in \Omega_0 \setminus S_{\theta} : x + y \leq 2 \theta \} \ \subset \ B_0. $$
\end{enumerate}
\end{theorem}
\noindent If the inequality \eqref{eq:dispersal} holds,  we can consider that the system has a very strong dispersal.
 Then Theorem \ref{th:dispersal} indicates that, when $\theta < 1/2$, both patches synchronize fast enough so that the global dynamics
 reduce to the one of a single-patch model: if the initial global density, i.e., the average of the densities in both patches, is
 below the Allee threshold then the population goes extinct whereas if it exceeds the Allee threshold then the population expands globally.
 In addition, the theoretical results in Theorems \ref{th:basin} and \ref{th:dispersal}, suggest that the smaller the Allee threshold,
 the smaller the basin of attraction of the extinction state and the larger the basin of attraction of the expansion state.
 This agrees with the simulation results of Figure \ref{fig:ODE}.

\indent Finally, in order to explore the number of locally stable equilibriums when the dispersal is small, we now look at the
 nullclines of the system. Define
 $$ f (x) \ = \ x \ - \ \frac{r x (x - \theta)(1 - x)}{\mu}. $$
 Then, the nullclines of the system \eqref{d_x}-\eqref{d_y} are given by $x = f(y)$ and $y = f(x)$.
 The interior equilibriums are determined by the positive roots of $x = f (f (x))$, which is a polynomial with degree 9.
 This implies that the system has at most 8 interior equilibriums since 0 is always a solution.

\indent According to the expression of the nullclines $y = f (x)$ and $x = f (y)$ (see Figure \ref{fig:nullclines}
 page \pageref{fig:nullclines}), we can see that the number of interior equilibriums strongly depends upon the value of the
 dispersal parameter: in the presence of strong dispersal, both patches synchronize and the system has only two positive interior
 equilibriums $E_{\theta}$ and $E_1$, which is confirmed by Theorem \ref{th:dispersal}, while in the presence of weak dispersal,
 there is enough independence between both patches so that the system has 8 positive interior equilibriums.
 We are interested in the locally stable equilibriums since the possible number of attractors is intimately connected to the number
 of stable equilibriums.
 The following theorem summarizes the properties of the equilibriums and their stability for different parameters' values.

\begin{theorem}[Multiple attractors]
\label{th:attractors}
\begin{enumerate}
\item If $r > 0$, $\theta, \mu \in [0, 1]$ then every trajectory converges in $[0, 1]^2$ so all the equilibriums $(x^*, y^*) \in [0, 1]^2$. \vspace{4pt}
\item If $6 \mu > r (\theta^2 - \theta + 1)$ then there are only three equilibriums: $E_0$, $E_{\theta}$ and $E_1$, with
 $E_0$ and $E_1$ locally stable and $E_{\theta}$ saddle.\vspace{4pt}
\item If $4 \mu < r (1 - \theta)^2$ then the nullcline
 $$ y \ = \ f (x) \ := \ x \ - \ \frac{rx (x - \theta) (1 - x)}{\mu} $$
 has exactly two positive roots that we denote by $0 < x_1 < x_2$.
 Let $M = \max_{0 \leq x \leq x_1} f (x)$. \vspace{4pt}
\begin{enumerate}
\item If $x_1 < M < x_2$ then the system has five fixed points with only two locally stable: $E_0$ and $E_1$. \vspace{4pt}
\item If $M \geq 1$ then the system achieves its maximum number of equilibriums which is equal to 9;
 only four of them are locally stable: two symmetric equilibriums $E_0$ and $E_1$ and two asymmetric interior equilibriums $(x_s, y_s)$ and $(y_s, x_s)$.
\end{enumerate}
\end{enumerate}
\end{theorem}

\begin{figure}[h!]
\centering
\scalebox{0.38}{\input{ODE-1.pstex_t}}
\caption{Solution curves and basins of attraction of the deterministic model with $r = 1$ and $\theta<1/2$.
 The values of the Allee threshold and fraction parameter, $\theta$ and $\mu$, are indicated at the bottom of each
 simulation pictures.
 The dark dots are locally stable equilibriums.
 The solid lines are trajectories with arrows pointing to its converging state.
 The dashed line is the straight line: $x + y = 2 \theta$.
 The grey region is the basin attraction of the expansion state $E_1$.
 The white region is the basin attraction of $(x_s, y_s)$ and $(y_s, x_s)$.
 The dark grey region is the basin attraction of the extinction state $E_0$. Based on Theorem \ref{th:attractors}, it is enough
 to restrict the system \eqref{d_x}-\eqref{d_y} to the compact space $[0,1]^2$.}
\label{fig:ODE}
\end{figure}
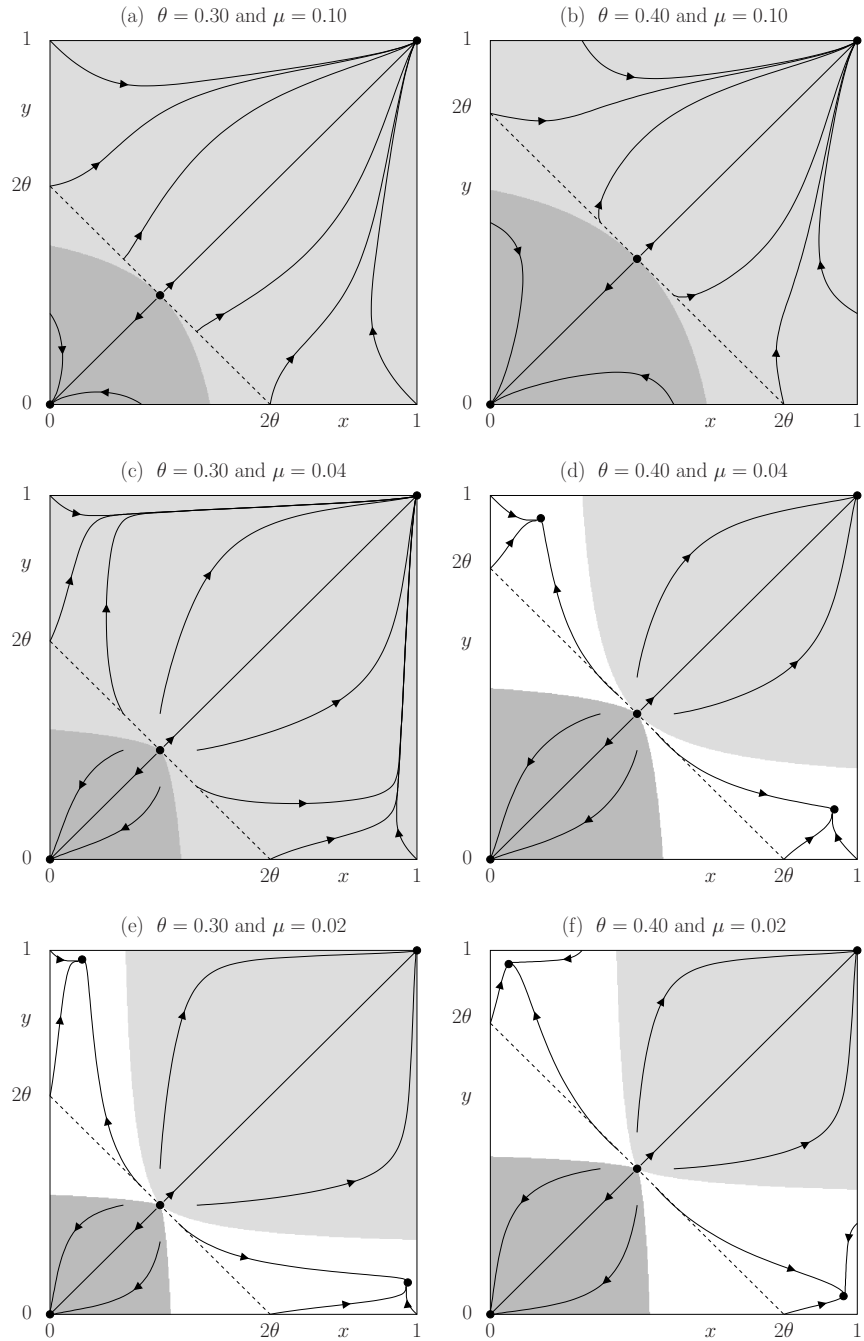

\clearpage

\begin{figure}[h!]
\centering
\scalebox{0.38}{\input{ODE-2.pstex_t}}
\caption{Solution curves and basins of attraction of the deterministic model with $r = 1$ and $\theta>1/2$.
 The values of the Allee threshold and fraction parameter, $\theta$ and $\mu$, are indicated at the bottom of each
 simulation pictures.
 The dark dots are locally stable equilibriums.
 The solid lines are trajectories with arrows pointing to its converging state.
 The dashed line is the straight line: $x + y = 2 \theta$.
 The grey region is the basin attraction of the expansion state $E_1$.
 The white region is the basin attraction of $(x_s, y_s)$ and $(y_s, x_s)$.
 The dark grey region is the basin attraction of the extinction state $E_0$.}
 \label{fig:ODE2}
\end{figure}
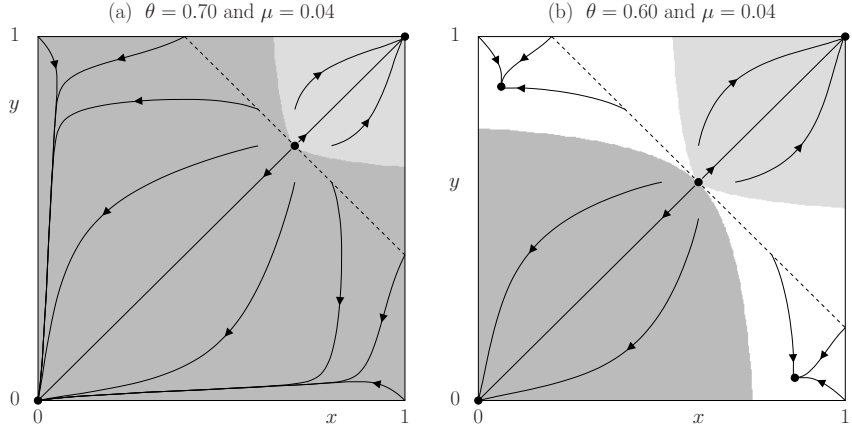

\noindent Part 1 of Theorem \ref{th:attractors} suggests that we can restrict our analysis of the basin of attraction of locally
 stable equilibriums to the compact space $[0,1]^2$.
 Moreover, from Theorem \ref{th:attractors}, we can see that when the dispersal parameter is small enough, the system has 9
 equilibriums, including four locally stable equilibriums.
 In the presence of an Allee effect, small dispersal may promote survival:
 patches that are below the Allee threshold are rescued by immigrants from adjacent patches above the Allee threshold.
 This implies that when dispersal is introduced to a system with an Allee effect, populations can exist at intermediate densities,
 corresponding to the equilibriums $(x_s, y_s)$ and $(y_s, x_s)$, as a source-sink system, or expand to high density $E_1$.
 Moreover, according to perturbation theory (Simon 1974; Amarasekare 2000), both asymmetric interior equilibriums appear
 from the equilibriums $(0, 1)$ and $(1, 0)$ of the uncoupled system, i.e., in the absence of dispersal, caused by the small
 perturbation $\mu$.
 Therefore, we have $x_s = O (\mu)$ and $y_s = 1 - O (\mu)$.
 Finally, note that the absence of limit cycles given by Theorem \ref{th:simple} when \eqref{simple_values} holds combined
 with Theorem \ref{th:attractors} implies that
\begin{corollary}[Four attractors]
\label{c:4a}
 If the system \eqref{d_x}-\eqref{d_y} has four locally stable equilibriums and inequality \eqref{simple_values} holds for some
 $c \in [0, 3)$, then the system has exactly four attractors.
\end{corollary}
\noindent When the system has four attractors as stated in Corollary \ref{c:4a}, the simulations in Figure \ref{fig:ODE} suggest
 that the smaller the dispersal, the smaller the basin of attraction of the extinction state and the expansion state, but the
 larger the basin attraction of the asymmetric interior equilibriums.
 In particular, if $\mu \rightarrow 0$, then
 $$ B_0 \rightarrow \Omega_{0,\theta}, \qquad B_1 \rightarrow \Omega_\theta, \qquad
    B_s \rightarrow \R^2_+ \setminus (\Omega_{0, \theta} \cup \Omega_\theta) $$
 where $B_s$ denotes the basin of attraction of the asymmetric interior equilibriums.


\subsection{ Simulations and Summary}

\indent Theorem \ref{th:dispersal} suggests that the larger the Allee threshold, the larger the basin of attraction of the extinction
 state and the smaller the basin of attraction of the expansion state when inequality \eqref{eq:dispersal} holds.
 The simulations shown in Figure \ref{fig:ODE} confirm this and give us a more complete picture of how the dispersal $\mu$
 and Allee threshold $\theta$ affect the exact basin of attraction of the locally stable equilibriums including asymmetric
 interior equilibriums:
\pagebreak
\begin{description}
\item  {\bf Effects of dispersal $\mu$} -- Fix Allee threshold $\theta$ and growth rate $r$, let dispersal $\mu$ vary. \vspace{4pt}
\begin{enumerate}
\item When $\mu$ is small so that the system has four locally stable equilibriums ($\mu$ smaller than some critical value $\mu_c$),
 the smaller the dispersal, the smaller the basin of attraction of the extinction state and the expansion state, but the larger the
 basin of attraction of the asymmetric interior equilibriums (see (d) and (f) of Figure \ref{fig:ODE}).
 This indicates that smaller dispersals promote persistence of the populations in both patches by creating sink-source dynamics. \vspace{4pt}
\item When $\mu$ is large so that the system has only two attractors $E_0$ and $E_1$ ($\mu$ larger than the critical value $\mu_c$),
 the larger the dispersal, the larger the basin of attraction of the extinction state but the smaller the basin of attraction of the
 expansion state when $\theta < 1/2$ (see (a) and (c) of Figure \ref{fig:ODE}).
 When $\theta > 1/2$, the monotonicity is flipped due to the symmetry of the system (see Figure \ref{fig:ODE2}). \vspace{4pt}
\item Extreme cases: when $\mu$ is very small, the two-patch model behaves nearly like the uncoupled system, having four attractors
 and almost the same basins of attraction, while when $\mu / r \to \infty$, the global population behaves according to a one-patch
 system with Allee threshold $2 \theta$, in particular
 $$ B_0 \ \longrightarrow \ \{(x, y) \in \R^2_+ : x + y < 2 \theta \}. $$
\end{enumerate}
\item {\bf Effects of Allee threshold $\theta$} -- Fix dispersal $\mu$ and growth rate $r$, let Allee threshold $\theta$ vary. \vspace{4pt}
\begin{enumerate}
\item[] Regardless of the number of locally stable equilibriums, the larger the Allee threshold, the larger the basin of attraction of the
 extinction state but the smaller the basin of attraction of the expansion state (see (a), (b), (e) and (f) of Figure \ref{fig:ODE}). \vspace{4pt}
\end{enumerate}
\item {\bf Effects of Growth rate $r$} -- Fix dispersal $\mu$ and Allee threshold $\theta$, let growth rate $r$ vary. \vspace{4pt}
\begin{enumerate}
\item[] By introducing the new time $\tau = t / r$, we can scale off the parameter $r$ of the
 system \eqref{d_x}-\eqref{d_y} so the dispersal $\mu$ becomes $\mu / r$.
 This implies that the growth rate $r$ and the dispersal parameter $\mu$ have opposite effects on the basin of attraction of the locally
 stable equilibriums, i.e., increasing the value of $r$ is equivalent to decreasing the value of $\mu$.
\end{enumerate}
\end{description}
\noindent Tables \ref{tab:comparisondeter1}-\ref{tab:comparisondeter2} give a complete picture, based on our analytical and numerical
 results, of how dispersal and Allee threshold affect the basin of attraction of the locally stable equilibriums.
 We only focus on the case $\theta < 1/2$ but similar results can be deduced when $\theta > 1/2$ using the symmetry of the system \eqref{d_x}-\eqref{d_y}.
\begin{table}[h]\begin{center}
\caption{\upshape{Summary for the deterministic model when $\theta < 1/2$ and variations of the parameters are restricted to the case when
 the system has only two attractors $E_0$ and $E_1$.}}
\label{tab:comparisondeter1}
\begin{tabular}{|c|c|c|}\hline
\multicolumn{3}{|c|}{Two attractors and $\theta < 1/2$} \\
\hline
 Parameters                        & Basin of attraction of $E_0$ & Basin of attraction of $E_1$ \\ \hline
 Dispersal $\mu \uparrow$          & $B_0$ $\uparrow$             & $B_1$ $\downarrow$ \\ \hline
 Allee threshold $\theta \uparrow$ & $B_0$ $\uparrow$             & $B_1$ $\downarrow$ \\ \hline
\multicolumn{3}{|c|}{In particular, if $\mu / r \rightarrow \infty$ then $B_0 \rightarrow \{(x,y) \in \R^2_+: x + y \leq 2 \theta\}$.} \\
\hline
\end{tabular}
\end{center}
\end{table}
\begin{table}[h]
\begin{center}
\caption{\upshape{Summary for the deterministic model when $\theta < 1/2$ and variations of parameters are restricted to the case when
 the system has four attractors: $E_0,E_1$ and $(x_s, y_s), (y_s, x_s)$.}}
\label{tab:comparisondeter2}
\begin{tabular}{|c|c|c|c|}
\hline \multicolumn{4}{|c|}{Four attractors and $\theta < 1/2$} \\
\hline
 Parameters & Basin of attr. of $E_0$ & Basin of attr. of asymmetric equilibriums & Basin of attr. of $E_1$ \\ \hline
 Dispersal $\mu \downarrow$ & $B_0 \downarrow$ & $B_s \uparrow$ & $B_1 \downarrow$ \\ \hline
 Allee threshold $\theta \downarrow$ & $B_0 \downarrow$ & no monotonicity & $B_1 \uparrow$ \\ \hline
 $\mu \rightarrow 0$ & $B_0 \rightarrow \Omega_{0,\theta}$ & $B_s \rightarrow \R^2_+ \setminus (\Omega_{0,\theta} \cup \Omega_\theta)$ & $B_1 \rightarrow \Omega_{\theta}$ \\
\hline
\end{tabular}
\end{center}
\end{table}



\section{ Definition of the stochastic model and main results}
\label{sec:stochastic}

\indent While the deterministic model is similar to the one in Ackleh \emph{et al} (2007), our stochastic model differs from theirs,
 which is derived naturally from the deterministic model by including independent Poisson increments, i.e., variability in birth, death
 and migration events.
 This gives rise to a multi-patch individual-based model for which they study numerically the probability of a successful invasion, defined
 as the event that the population size in one patch exceeds some denominated threshold.
 However, well-known results about irreducible Markov chains imply that the population is driven almost surely to extinction which corresponds
 to the unique absorbing state of their stochastic process.
 In contrast, we model stochastically the two-patch system via a process that has two absorbing states corresponding to a global extinction
 and a global expansion, respectively.
 This allows to have a definition of successful invasion more rigorous and more tractable mathematically.
 In particular, while their stochastic model is designed to study numerically the probability that a population starting near the Allee
 threshold in each patch gets successfully established, our model is designed to study analytically the probability that a fully occupied
 patch successfully invade a nearby empty patch.
 More precisely, to understand the effect of stochasticity on the interactions between both patches, we introduce a Markov jump process
 that, similarly to the deterministic model, keeps track of the evolution of the population size in each patch.
 To obtain a Markov process, the state is updated at random times represented by the points of a Poisson process with a certain intensity
 making the times between consecutive updates independent exponentially distributed random variables.
 Motivated by the fact that the unit square $S = [0, 1]^2$ is positive invariant for the deterministic model, we will choose this set as the
 state space, i.e., the state at time $t$ is a random vector $\eta_t = (X_t, Y_t) \in S$, where the first and
 second coordinates represent the population size in the first and second patch, respectively.
 Following the deterministic model, the stochastic dynamics involve three mechanisms: expansion, extinction, and migration.
 To model the presence of an Allee affect, we again introduce a threshold parameter $\theta \in (0, 1)$ that can be seen as a critical size
 under which the population undergoes extinction and above which the population undergoes expansion, i.e., Allee threshold.
 This aspect is modeled by assuming that each component of the stochastic process jumps independently at rate $r > 0$ to
 either 0 (extinction) or 1 (expansion) depending on whether it lies below or above the Allee threshold.
 Recall that an event ``happens at rate $r$'' if the probability that it happens during a short time interval of
 length $\Delta t$ approaches $r (\Delta t)$ as $\Delta t \ \to \ 0$.
 In particular, expansion and extinction are formally described by the conditional probabilities
 $$ \begin{array}{l}
  P \,(X_{t + \Delta t} = 1 \ | \ X_t > \theta) \ = \
  P \,(Y_{t + \Delta t} = 1 \ | \ Y_t > \theta) \ = \ r \Delta t + o (\Delta t) \vspace{4pt} \\
  P \,(X_{t + \Delta t} = 0 \ | \ X_t < \theta) \ = \
  P \,(Y_{t + \Delta t} = 0 \ | \ Y_t < \theta) \ = \ r \Delta t + o (\Delta t). \end{array} $$
 This is also equivalent to saying that the waiting time for an expansion or an extinction is exponentially
 distributed with mean $1 / r$.
 Given that the population size in a given patch is at the Allee threshold, we flip a fair coin to decide whether an
 expansion or an extension event occurs at that patch which, in view of well-known properties of Poisson processes,
 implies that
 $$ \begin{array}{l}
  P \,(X_{t + \Delta t} = 1 \ | \ X_t = \theta) \ = \
  P \,(Y_{t + \Delta t} = 1 \ | \ Y_t = \theta) \ = \ (r / 2) \Delta t + o (\Delta t) \vspace{4pt} \\
  P \,(X_{t + \Delta t} = 0 \ | \ X_t = \theta) \ = \
  P \,(Y_{t + \Delta t} = 0 \ | \ Y_t = \theta) \ = \ (r / 2) \Delta t + o (\Delta t). \end{array} $$
 To understand the effects of inter-patch interactions on the evolution of the system, we also include migration events
 consisting of the displacement of a fraction $\mu$ of the population of each patch to the other patch.
 We assume that these events occur at the normalized rate 1, therefore migrations are described by
 $$ P \,((X_{t + \Delta t}, Y_{t + \Delta t}) = (1 - \mu) \,(X_t, Y_t) + \mu \,(Y_t, X_t)) \ = \ \Delta t + o (\Delta t), $$
 We refer to Figure \ref{fig:dynamics} for a schematic illustration of the dynamics, where dark rectangles represent parts
 of the populations which are interchanged in the event of a migration. 
 To analyze mathematically the stochastic process, it will be useful to look at the model as a simple example of interacting
 particle system.
 Interacting particle systems are continuous-time Markov processes whose state space maps the vertex set of a connected
 graph into a set representing the possible states at each vertex.
 The evolution is described by local interactions as the rate of change at a given vertex only depends on the configuration
 in its neighborhood.
 In particular, the Markov process $\{\eta_t \}_t$ can be seen as an interacting particle system evolving on a very simple
 graph that consists of only two vertices, representing both patches, connected by one edge, indicating that patches interact.
 The reason for looking at the stochastic model as an example of interacting particle system is that this will allow us to
 construct the process graphically from a collection of independent Poisson processes based on an idea of Harris (1972),
 which is a powerful tool to analyze the process mathematically.

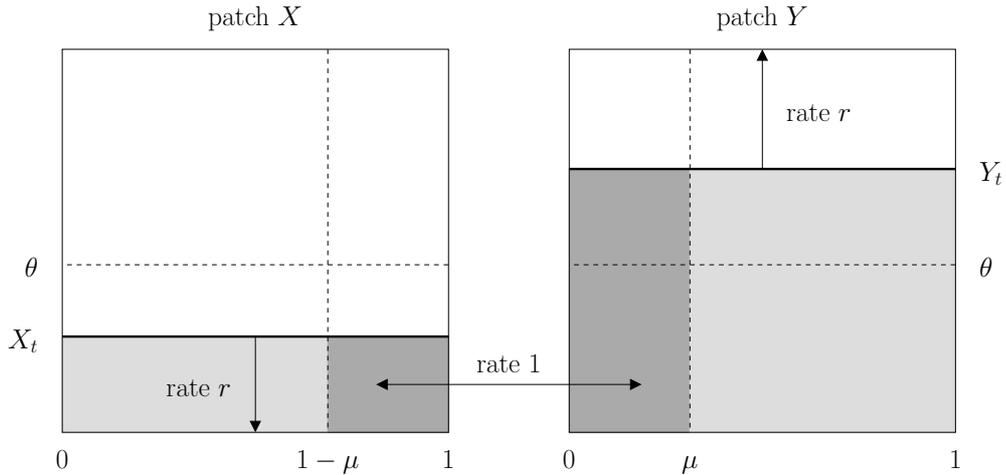
\begin{figure}[t]
\centering
\scalebox{0.50}{\input{dynamics.pstex_t}}
\caption{Schematic representation of the stochastic model $\eta_t = (X_t, Y_t)$.
 The dark rectangles represent parts of the populations which are exchanged in the event of a migration.}
\label{fig:dynamics}
\end{figure}

\indent We now describe in details the behavior of the process along with our main results.
 Note that, considering a stochastic model rather than a deterministic one, the long-term behavior is
 described by a set of invariant measures on the state space rather than single point equilibriums.
 To the two trivial equilibriums of the deterministic model, $E_0$ and $E_1$, correspond two invariant measures which are
 Dirac measures that concentrate on those two points, respectively.
 These two measures are two absorbing states: the configuration in which both patches are empty and the configuration in
 which both patches are fully occupied.
 We call global extinction and global expansion the events that the process eventually fixates to the first and the second
 absorbing state, respectively.
 Interestingly, to the two asymmetric equilibriums of the deterministic model in the presence of weak dispersal correspond
 two quasi-stationary distributions representing two metastable states of the stochastic process (see
 Theorem \ref{metastable}): depending on the initial configuration, the transient behavior might be described by one of
 these two quasi-stationary distributions, but after a long random time in the presence of weak dispersal (see
 Theorem \ref{metastability}), the system fixates to one of the two absorbing states, suggesting that situations predicted
 by the deterministic model in which a small population can live next to a large population are artificially stable.
 Another important question is how stochasticity affects the geometry of the basins of attraction of the two absorbing
 states, although strictly speaking there is no basin of attraction for the stochastic model since the limiting behavior
 might be unpredictable, and how fast the system fixates.
 We will see that there is a set of initial configurations for which the limiting behavior of the stochastic process is
 predictable, and fixation to one of the two absorbing states occurs quickly (see Theorem \ref{fixation}).
 Starting from any other configuration, the limiting behavior becomes unpredictable in the sense that the process may
 reach any of the two absorbing states with positive probability.
 In the presence of weak dispersal, however, the limit is almost predictable in the sense that the probability that the
 system undergoes a global expansion after exiting one of its metastable states approaches zero or one (see
 Theorem \ref{probabilities}).
 Whether the system fixates to one or the other absorbing state strongly depends on the value of the Allee threshold.
 The limit is less and less predictable and the time to fixation shorter and shorter as the dispersal parameter increases.



\subsection{ Predictable behavior}

\indent In order to describe rigorously the behavior of the stochastic model introduced above, our main objective is to
 estimate the times to fixation
 $$ \tau^+ = \inf \,\{t \geq 0 : X_t = Y_t = 1 \} \quad \hbox{and} \quad
    \tau^- = \inf \,\{t \geq 0 : X_t = Y_t = 0 \}, $$
 and the corresponding probabilities of fixation,
 $$ P \,(\tau = \tau^+) \quad \hbox{and} \quad
    P \,(\tau = \tau^-) \quad \hbox{where} \quad \tau = \min (\tau^+, \tau^-), $$
 as a function of the initial configuration and the three parameters of the system.
 As previously explained, in contrast with the deterministic model which can have up to four distinct attractors, with
 probability one, either global expansion or global extinction occurs for the stochastic process, i.e.,
 $$ P \,(\tau < \infty) \ = \ P \,(\tau = \tau^+) + P \,(\tau = \tau^-) \ = \ 1. $$
 The state space can be divided into four subsets.
 Starting from only two of these subsets the limit is predictable in the sense that
 $$ P \,(\tau = \tau^+) \ \in \ \{0, 1 \}. $$
 We call an upper configuration any configuration of the system in which the population size in each patch exceeds the Allee
 threshold, and a lower configuration any configuration in which the population size in each patch lies below the Allee threshold.
 These sets are denoted respectively by
 $$ \begin{array}{rcl}
    \Omega^+ & = & \{(x, y) \in S : x > \theta \ \hbox{and} \ y > \theta \} \ = \
    \Omega_{\theta, 1} \setminus \{(x, y) : x = \theta \ \hbox{or} \ y = \theta \} \vspace{4pt} \\
    \Omega^- & = & \{(x, y) \in S : x < \theta \ \hbox{and} \ y < \theta \} \ = \
    \Omega_{0, \theta} \setminus \{(x, y) : x = \theta \ \hbox{or} \ y = \theta \}. \end{array} $$
 Note that the set of upper configurations is closed under the dynamics, i.e., once the system hits an upper configuration,
 the configuration at any later time is also an upper configuration.
 This implies that, starting from an upper configuration, global expansion occurs with probability one.
 Similarly, starting from a lower configuration, global extinction occurs with probability one.
 By representing the process graphically, the time to fixation can be computed explicitly, as stated in the following theorem.
\begin{theorem}[time to fixation]
\label{fixation}
 We have
 $$ \E \,[\,\tau^+ \,| \ (X_0, Y_0) \in \Omega^+] \ = \
    \E \,[\,\tau^- \,| \ (X_0, Y_0) \in \Omega^-] \ = \ \frac{6r + 1}{2 r^2}. $$
\end{theorem}
 The previous theorem indicates that, starting from an upper configuration, the system converges with
 probability one to the absorbing state $(1, 1)$, whereas starting from a lower configuration, it converges
 with probability one to the other absorbing state $(0, 0)$.
 This result can be seen as the analog of Theorem \ref{th:basin} which states that the sets of upper and
 lower configurations are included in the basin of attraction of the equilibrium points $E_1$ and $E_0$, respectively.
 Theorem \ref{fixation} also indicates that, when the rates at which expansions, extinctions, and migrations
 occur are of the same order, the expected time to fixation is quite short.



\subsection{ Metastability}

\indent The long-term behavior of the process starting from a configuration which is neither an upper configuration nor
 a lower configuration is more difficult to study as the probabilities of global expansion and global extinction are both
 strictly positive, which we shall refer to as unpredictable behavior.
 We will prove that, in any case, the system hits either an upper or a lower configuration at a random time which is almost
 surely finite, after which it evolves as indicated by Theorem \ref{fixation}.
 Hence, the time to fixation and probabilities of global expansion and extinction can be determined by estimating the
 hitting times
 $$ T^+ = \inf \,\{t \geq 0 : (X_t, Y_t) \in \Omega^+ \} \quad \hbox{and} \quad
    T^- = \inf \,\{t \geq 0 : (X_t, Y_t) \in \Omega^- \} $$
 and the corresponding hitting probabilities
 $$ P \,(T = T^+) \quad \hbox{and} \quad
    P \,(T = T^-) \quad \hbox{where} \quad T = \min (T^+, T^-), $$
 since Theorem \ref{fixation} implies that
 $$ \E \,[\tau] \ = \ \E \,[T] \ + \ \frac{6r + 1}{2r^2} \quad \hbox{and} \quad
  P \,(\tau = \tau^+) \ = \ P \,(T = T^+). $$
 Even though our next results hold for any values of the parameters, they indicate that interesting behaviors emerge
 when the dispersal parameter $\mu$ is small.
 In contrast with the deterministic model which, in this case, has four attractors, as indicated
 by Theorem \ref{th:attractors}, the stochastic model first exhibits a metastable behavior by oscillating for an arbitrarily
 long time around one of the two nontrivial equilibriums of the deterministic model, and then fixates to one of its two
 absorbing states.
 The limit is almost predictable as the probability of global expansion approaches either 0 or 1 depending on the
 value of the threshold parameter.
 For simplicity and since the system is symmetric, we shall assume that $X_0 = 0$ and $Y_0 = 1$ but the proofs of our results
 easily extend to the more general case when
 $$ 0 \ < \ \mu \ \ll \ \min \{|X_0 - \theta|, |Y_0 - \theta| \}. $$
 Recall that, starting from an upper configuration or a lower configuration, the time to fixation is rather small.
 In contrast, when $X_0 = 0$, $Y_0 = 1$ and $\mu$ is small, the stochastic process converges to a quasi-stationary distribution
 in which the population size at patch $X$ is relatively close to 0 and the population size at patch~$Y$ relatively close to 1,
 and stays at its quasi-stationary distribution for a very long time, i.e., the expected value of $T$ is large.
 However, due to stochasticity, the system reaches eventually an upper or a lower configuration, and then fixates rapidly.
 The next theorem gives an explicit lower bound of the expected value of the hitting time, which is the time the
 system stays at its quasi-stationary distribution.
\begin{theorem}[metastability]
\label{metastability}
 For any initial configuration, we have
 $$ P \,(T < \infty) \ = \ P \,(\tau < \infty) \ = \ 1. $$
 Moreover, $$ \E \,[\,T \,| \,(X_0, Y_0) = (0, 1)] \ \geq \ \frac{n_0}{2 + 4r} \ \bigg(\frac{1 + 2r}{1 + r} \bigg)^{n_0} $$
 where $$ n_0 \ = \ \frac{1}{2} \, \bigg\lfloor \frac{\min (\ln (1 - \theta), \ln (\theta))}{\ln (1 - \mu)} \bigg\rfloor. $$
\end{theorem}
 Note that, when the Allee threshold is bounded away from 0 and 1, and the dispersal parameter is small, $n_0$ is large, and so
 is the expected value of the hitting time $T$.
 Note also that, before the hitting time, no expansion event can occur at patch $X$ while no extinction event can occur at patch $Y$.
 This indicates that the metastable state of the stochastic two-patch model is described by the stationary distribution of the
 Markov process $\bar \eta_t = (\bar X_t, \bar Y_t)$ with state space $S = [0, 1]^2$, and whose evolution is given by
 $$ \begin{array}{rcl}
  P \,(\bar X_{t + \Delta t} = 0 \ | \ \bar X_t \neq 0) \ = \
  P \,(\bar Y_{t + \Delta t} = 1 \ | \ \bar Y_t \neq 1) & = & r \Delta t + o (\Delta t) \vspace{4pt} \\
  P \,((\bar X_{t + \Delta t}, \bar Y_{t + \Delta t}) = (1 - \mu) \,(\bar X_t, \bar Y_t) + \mu \,(\bar Y_t, \bar X_t)) & = & \Delta t + o (\Delta t), \end{array} $$
 where $\Delta t$ is a small time interval.
 That is, the process $\{\bar \eta_t \}_t$ is obtained from $\{\eta_t \}_t$ by assuming that only extinction events at patch $X$
 and only expansion events at patch $Y$ can occur, which indeed describes the evolution of the original process $\{\eta_t \}_t$
 before it reaches an upper or a lower configuration.
 Letting $\nu$ denote the stationary distribution of this new process, the behavior of the stochastic two-patch model before
 the hitting time $T$ is described by the following theorem.
\begin{theorem}[metastable state]
\label{metastable}
 Under the measure $\nu$ we have
 $$ \E_{\nu} \,(\bar X_t) \ \leq \  \frac{\mu}{r + \mu} \qquad \hbox{and} \qquad
    \E_{\nu} \,(\bar Y_t) \ \geq \ 1- \frac{\mu}{r + \mu}. $$
\end{theorem}
 This indicates that, when $\mu$ is small, the population size at patch $X$ is close to 0 (i.e., $O (\mu)$) and the population size
 at patch $Y$ close to 1 (i.e., $1- O (\mu)$).
 The expected values above have to be thought of as the analog of the two asymmetric equilibriums of the deterministic model:
 $(x_s, y_s)$ and $(y_s, x_s)$.
 After evolving a long time according to the quasi-stationary distribution $\nu$, the process hits either an upper
 or a lower configuration, so the last question we would like to answer is whether global expansion or global extinction occurs
 after the system exits its metastable state.
 Starting from an upper or a lower configuration, the answer is given by Theorem \ref{fixation}.
 Starting from $X_0 = 0$ and $Y_0 = 1$, the symmetry of the model implies that
 $$ P \,(\tau = \tau^+) \ = \ P \,(\tau = \tau^-) \ = \ 1/2 \quad \hbox{whenever} \ \theta = 1/2. $$
 Our last result shows that, when $\theta \neq 1/2$ and $\mu > 0$ is small, the limiting behavior of the system is almost
 predictable in the sense that the probability of global expansion approaches either 0 or 1.
\begin{theorem}[hitting probabilities]
\label{probabilities}
 Assume that $\theta < 1/2$. Then
 $$ \frac{P \,(T = T^- \ | \ (X_0, Y_0) = (0, 1))}{P \,(T = T^+ \ | \ (X_0, Y_0) = (0, 1))} \ \leq \ \bigg(\frac{1}{1 + r} \bigg)^{m_0} $$
 where $$ m_0 \ = \ \bigg\lfloor \frac{\ln (2 \theta)}{\ln (1 - \mu)} \bigg\rfloor. $$
\end{theorem}
 The previous theorem indicates that, when $\mu > 0$ is small,
 $$ \begin{array}{l}
  P \,(\tau = \tau^+ \ | \ (X_0, Y_0) = (0, 1)) \ = \
  P \,(T = T^+ \ | \ (X_0, Y_0) = (0, 1)) \vspace{4pt} \\ \hspace{25pt} = \
  1 - P \,(T = T^- \ | \ (X_0, Y_0) = (0, 1)) \ \geq \ 1 - (1 + r)^{- m_0} \ \approx \ 1. \end{array} $$
 In particular, in contrast with the deterministic model for which the limit depends on the initial condition and the geometry of the
 basins of attraction, starting from any initial configuration but an upper or a lower configuration, the limiting behavior of the
 stochastic model is only sensitive to the value of the parameters, with the Allee threshold $\theta$ playing a central role.


\subsection{ Simulation results}

\indent While Theorem \ref{fixation} gives an exact estimate of the time to fixation starting from particular initial conditions,
 the other results provide theoretical lower and upper bounds that allows us to gain a valuable insight into the long-term behavior
 of the stochastic two-patch model in the presence of weak dispersal.
 To better understand the combined effect of the Allee threshold and dispersal parameter when starting from heterogeneous initial
 conditions, we refer the reader to the numerical simulations of Figure \ref{fig:proba-time} and Tables \ref{tab:proba}-\ref{tab:time}.
 The left panel of the figure represents the probability of a global extinction, with the probability increasing with the darkness,
 and the right panel the expected time to fixation, with time increasing with the darkness, as a function of the dispersal parameter
 and the Allee threshold.
 The tables provide some numerical values of the probability of extinction and expected time to fixation averaged over 10,000
 independent realizations of the stochastic process for specific values of the parameters.
 The predictions based on Theorems \ref{metastability} and \ref{probabilities} that the time to extinction blows up and the
 probability of extinction approaches either zero or one in the presence of weak dispersal appears clearly looking at the left side
 of both panels and the left column of the tables for which~$\mu = 0.02$.
 The left panel and Table \ref{tab:proba} further indicate that the probability of a global extinction depends non-monotonically
 upon the dispersal parameter:
 when the Allee threshold is below one half, the probability of extinction first increases with the dispersal parameter
 and then decreases after the dispersal reaches a critical value that depends on $\theta$, which can be easily seen in the
 row $\theta = 0.45$ of the table.
 When the Allee threshold exceeds one half, the monotonicity is flipped.
 Simulations also indicate that, the dispersal parameter being fixed, the probability of extinction increases as the Allee
 threshold increases.
 Although we omit the details of the proof, this can be easily shown analytically invoking a standard coupling argument to compare
 two processes, the first one with Allee threshold $\theta_1$ and the second one with $\theta_2 > \theta_1$, the other parameters
 being the same for both processes.
 The black triangle labeled 1 in the upper right corner of the left picture reveals that global extinction occurs almost surely
 when $\theta > 1 - \mu$.
 Indeed, starting from the heterogeneous condition $X_0 = 0$ and $Y_0 = 1$, after the first migration event, we have
 $$ X_t \ = \ \mu \quad \hbox{and} \quad Y_t \ = \ 1 - \mu \quad \hbox{and so} \quad \max \,(X_t, Y_t) \ = \ 1 - \mu \ < \ \theta. $$
 In particular, both patches are below the Allee threshold from which it follows that the population goes extinct eventually.
 Almost sure global expansion in the parameter region corresponding to the lower right white triangle labeled 2 can be proved similarly.
 Finally, as suggested by Theorem \ref{metastability}, the right picture and Table \ref{tab:time} indicate that the expected value of
 the time to fixation increases as the dispersal parameter decreases but also as the Allee threshold gets closer to one half, which
 can again be proved analytically based on standard coupling arguments even through we omit the details of the proof.

\begin{figure}[t]
\centering
\scalebox{0.60}{\input{proba-time.pstex_t}}
\caption{Simulation results for the probability of a global extinction and the time to fixation of the stochastic model starting with
 one empty patch and one fully occupied patch and with growth parameter $r = 0.25$.
 Left: the gradation of grey represents the probability of a global extinction ranging from 0 = white to 1 = black.
 Right: the gradation of grey represents the time to fixation ranging from 0 = white to 100 or more = black.
 In both pictures, the probability and time are computed from the average of 10,000 independent simulation runs for 200 different
 values of the Allee threshold ranging from 0.25 to 0.75.
 These are further computed for 190 different values of the dispersal parameter ranging from 0.02 to 0.50, and 76 different
 values of the dispersal parameter ranging from 0.02 to 0.20, respectively.}
\label{fig:proba-time}
\end{figure}
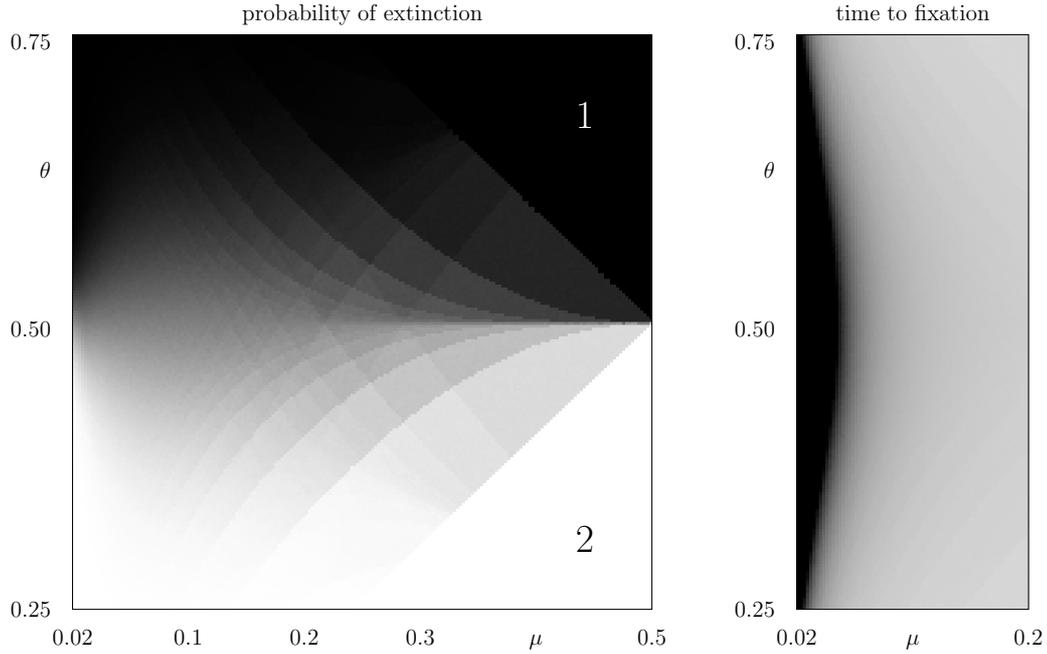

\begin{table}[h!]
\begin{center}
\caption{\upshape{Probability of extinction ($r = 0.25$)}}
\label{tab:proba}
\begin{tabular}{|c|c|c|c|c|c|c|c|} \hline
 &  0.020 &  0.050 &  0.100 &  0.200 &  0.300 &  0.400 &  0.500 \\ \hline
$\theta = 0.75$ & 
  1.000 &  0.999 &  0.995 &  0.999 &  1.000 &  1.000 &  1.000 \\ \hline
$\theta = 0.70$ & 
  1.000 &  0.993 &  0.978 &  0.991 &  0.972 &  1.000 &  1.000 \\ \hline
$\theta = 0.65$ & 
  0.999 &  0.973 &  0.928 &  0.928 &  0.944 &  1.000 &  1.000 \\ \hline
$\theta = 0.60$ & 
  0.992 &  0.916 &  0.834 &  0.823 &  0.905 &  0.852 &  1.000 \\ \hline
$\theta = 0.55$ & 
  0.910 &  0.759 &  0.681 &  0.719 &  0.779 &  0.857 &  1.000 \\ \hline
$\theta = 0.50$ & 
  0.439 &  0.506 &  0.496 &  0.502 &  0.502 &  0.488 &  0.000 \\ \hline
$\theta = 0.45$ & 
  0.056 &  0.247 &  0.316 &  0.288 &  0.230 &  0.141 &  0.000 \\ \hline
$\theta = 0.40$ & 
  0.004 &  0.083 &  0.163 &  0.176 &  0.099 &  0.000 &  0.000 \\ \hline
$\theta = 0.35$ & 
  0.000 &  0.030 &  0.066 &  0.078 &  0.062 &  0.000 &  0.000 \\ \hline
$\theta = 0.30$ & 
  0.000 &  0.007 &  0.022 &  0.011 &  0.000 &  0.000 &  0.000 \\ \hline
$\theta = 0.25$ & 
  0.000 &  0.001 &  0.007 &  0.002 &  0.000 &  0.000 &  0.000 \\ \hline
\end{tabular}
\end{center}
\end{table}

\begin{table}[h!]
\begin{center}
\caption{\upshape{Time to fixation ($r = 0.25$)}}
\label{tab:time}
\begin{tabular}{|c|c|c|c|c|c|c|c|} \hline
 &  0.020 &  0.050 &  0.100 &  0.200 &  0.300 &  0.400 &  0.500 \\ \hline
$\theta = 0.75$ & 
 152.711 & 28.888 & 19.749 & 16.133 & 15.165 & 14.994 & 15.058 \\ \hline
$\theta = 0.70$ & 
 334.855 & 38.060 & 21.694 & 16.526 & 16.291 & 15.183 & 14.960 \\ \hline
$\theta = 0.65$ & 
 815.171 & 52.247 & 24.436 & 17.526 & 16.608 & 15.027 & 15.164 \\ \hline
$\theta = 0.60$ & 
 2052.663 & 70.766 & 27.399 & 18.973 & 16.458 & 16.173 & 15.065 \\ \hline
$\theta = 0.55$ & 
 6058.586 & 91.957 & 29.515 & 19.258 & 16.877 & 16.092 & 14.820 \\ \hline
$\theta = 0.50$ & 
 10520.799 & 102.694 & 30.245 & 20.034 & 18.747 & 18.606 & 14.884 \\ \hline
$\theta = 0.45$ & 
 5566.830 & 91.369 & 29.453 & 19.086 & 16.838 & 15.946 & 14.907 \\ \hline
$\theta = 0.40$ & 
 2075.433 & 70.278 & 27.318 & 18.892 & 16.220 & 14.838 & 14.946 \\ \hline
$\theta = 0.35$ & 
 829.009 & 51.486 & 23.954 & 17.500 & 16.421 & 14.973 & 14.936 \\ \hline
$\theta = 0.30$ & 
 339.504 & 37.717 & 21.467 & 16.309 & 14.867 & 15.143 & 14.993 \\ \hline
$\theta = 0.25$ & 
 149.811 & 28.528 & 19.424 & 16.066 & 15.021 & 14.885 & 14.934 \\ \hline
\end{tabular}
\end{center}
\end{table}

\section{ Comparison and biological implications}
\label{sec:comparison}


\begin{table}[h]
\begin{center}
\caption{\upshape{Comparison between deterministic and stochastic models}}
\begin{tabular}{|c|c|c|}
\hline & Deterministic model & Stochastic model \\ \hline
 Dispersal parameter               & $\theta < 1/2$ \ \ \ $|$ \ \ \ $\theta > 1/2$ & $\theta < 1/2$ \ \ \ $|$ \ \ \ $\theta > 1/2$ \\ \hline
 No dispersal $\mu = 0$            & 4 attractors & 4 absorbing states \\ \hline
 Weak dispersal $\mu > 0$          & 4 attractors & 2 absorbing states \\
                                   & $B_0$ and $B_1 \uparrow$ as $\mu \uparrow \mu_c$ & + 2 metastable states \\
                                   & & $P \,(\hbox{expansion}) \approx 1$ \ $|$ \ $P \,(\hbox{extinction}) \approx 1$ \\ \hline
 Critical dispersal $\mu = \mu_c$  & Both patches synchronize & \\
                                   & 4 attractors $\to$ 2 attractors & 2 absorbing states \\ \hline
 Stronger dispersal $\mu > \mu_c$  & 2 attractors & 2 absorbing states \\
                                   & $B_0 \uparrow$ as $\mu \uparrow$ \ \ $|$ \ \ $B_1 \uparrow$ as $\mu \uparrow$ & unpredictability + quick fixation \\ \hline
 Very strong dispersal $\mu/r$ large  & same behavior as one-patch model & same behavior as one-patch model \\
                                   &  & when starting from $(0, 1)$ \\ \hline
\end{tabular}
\end{center}
\label{tab:comparison}
\end{table}

\noindent Recall that, in the absence of interactions between patches, both the deterministic model and the stochastic
 model predict a local expansion in patches where the initial population size is above the Allee threshold and a local extinction
 in patches where the initial population size is below the Allee threshold.
 This induces the existence of four locally stable equilibriums for the deterministic model, and four absorbing states for the
 stochastic model, which correspond to cases when the population in each patch either goes extinct or gets established.
 Including interactions between patches, our results for the deterministic model indicate that, in the presence of weak dispersal,
 the dynamics retain four attractors, just as in the absence of interactions, up to a critical value $\mu_c$ when the patches
 synchronize: the two asymmetric equilibrium points are lost so that only global expansion and global extinction can happen.
 In contrast, including both stochasticity and even weak interactions, only the two absorbing states corresponding to global
 expansion and global extinction are retained.
 The most interesting behaviors emerge when the dispersal is weak, in which case, to the two asymmetric locally stable equilibriums
 of the deterministic model, correspond two metastable states for the stochastic model.

\indent Looking at the global dynamics, the predictions based on the analysis of the deterministic two-patch model indicate that
 below the critical value $\mu_c$ dispersal promotes global expansion and global extinction in the sense that the basins of attraction
 of the two trivial fixed points expands while increasing the dispersal parameter.
 Above the critical value $\mu_c$ dispersal promotes a global expansion when the Allee threshold exceeds one half but promotes
 global extinction in the more realistic case when the Allee threshold lies below one half.
 As mentioned above, in the presence of weak dispersal, both asymmetric equilibrium points become two metastable states, i.e.,
 quasi-stationary distributions, after the inclusion of stochasticity, suggesting that situations in which a small population
 lives next to a large population are artificially stable: in such a context, the two-patch system evolves first as dictated
 by one of the two quasi-stationary distributions then, after a long random time, experiences either a global expansion or a
 global extinction.
 In addition, the long-term behavior of the stochastic model becomes almost predictable in the sense that, with very high
 probability, the system will undergo a global expansion when the Allee threshold lies below one half and a global extinction
 when the Allee threshold exceeds one half, which is of primary importance to predict the destiny of heterogeneous two-patch
 systems in the presence of weak dispersal.
 While increasing the dispersal parameter, the stochastic model no longer exhibits a metastable behavior, the time to fixation
 decreases, and the long-term behavior becomes more and more unpredictable.
 In the presence of a very strong dispersal, however, the analysis of the deterministic model and the stochastic model starting
 from a heterogeneous configuration give the same predictions.
 In this case, both patches synchronize enough so that the global dynamics reduce to that of a single-patch model:
 if the initial global density, i.e., the average of the densities in both patches, is below the Allee threshold then the
 population goes extinct whereas if it exceeds the Allee threshold then the population expands globally.

\indent Our analysis of idealized two-patch models is an important first step to understand more realistic multi-patch systems.
 Empirical data indicate that Allee thresholds in nature vary accross species and habitat types but are typically much smaller
 than one half.
 The predictions, based on the deterministic model in the presence of enough dispersal so that patches synchronize and on the
 stochastic model in the general case, that populations usually expand successfully when the Allee threshold is small is due
 to the fact that only two patches interact.
 Literally, the critical threshold $1/2$ has to be thought of as one divided by the number of patches.
 Looking at a multi-patch model in which $n$ patches interact all together, our analytical results suggest that a critical
 behavior should emerge for Allee thresholds near $1/n$ when starting with a population established in only one patch, and more
 generally the number of patches where the population is initially established divided by the number of interacting patches.
 Therefore, even for realistic values of the Allee threshold, the long-term behavior is no longer straightforward in the presence
 of a large number of patches.
 Numerical simulations can also provide a valuable insight into the long-term behavior of multi-patch models including additional
 refinements such as density-dependent dispersals, heterogeneous environments with possibly
 different Allee thresholds in different patches, and more importantly the inclusion of a spatial structure through a
 network of interactions represented by a two-dimensional regular lattice or more general planar graphs rather than a complete
 graph where patches interact all together.



\section{ Proofs}
\label{sec:proofs}

\subsection*{Preliminary results}

\noindent As previously explained, the key to proving our main results is to first identify a number of sets which are positive
 invariant for the system \eqref{d_x}-\eqref{d_y}.
 This will they give us means of decomposing the phase space by restricting our attention to the dynamics on each invariant set
 and then sewing together a global solution from the invariant pieces.
 Our first preliminary result indicates that, starting from any biologically meaningful initial condition, that is any condition
 belonging to $\Omega_0 := \R_+^2$, the trajectory of the system stays in the upper right quadrant and is bounded.

\begin{lemma}
\label{p_b}
 The system \eqref{d_x}-\eqref{d_y} is positive invariant and bounded in $\Omega_0$.
\end{lemma}
\begin{proof}
 Assuming by contradiction that the system \eqref{d_x}-\eqref{d_y} is not positive invariant in upper right quadrant,
 we can find $x_0, y_0 \geq 0$ and a time $T > 0$ such that
 $$ x (0) = x_0 \ \ \hbox{and} \ \ y (0) = y_0 \quad \hbox{implies} \quad (x (T), y (T)) \notin \Omega_0. $$
 Let $\Gamma$ denote the boundary of $\Omega_0$, i.e.,
 $$ \Gamma \ = \ \{(x, y) \in \R^2 : (x = 0 \ \hbox{and} \ y \geq 0) \ \hbox{or} \ (x \geq 0 \ \hbox{and} \ y = 0) \}. $$
 By continuity of the trajectories, the intermediate value theorem implies the existence of a time $t < T$ such that
 $(x (t), y (t)) \in \Gamma$ therefore
 $$ S \ := \ \sup \,\{t < T : (x (t), y (t)) \in \Gamma \} $$
 is well defined and $(x (t), y (t)) \notin \Omega_0$ for all $t \in (S, T]$.
 Then, we have the following alternative.
\begin{enumerate}
\item If $x (S) = y (S) = 0$ then $x (t) = y (t) = 0$ for all $t \geq S$, which contradicts the existence of $T$. \vspace{4pt}
\item If $x (S) = 0$ and $y (S) > 0$ then $x' (S) = \mu \,y (S) > 0$ so
 $$ \hbox{there exists} \ \ep > 0 \ \hbox{such that} \ x (t) > 0 \ \hbox{for all} \ t \in (S, S + \ep). $$
 This contradicts the existence of $S$. \vspace{4pt}
\item If $x (S) > 0$ and $y (S) = 0$, the same argument exchanging the roles of the functions $x$ and $y$ leads again to a contradiction.
\end{enumerate}
 In conclusion, if $x (0) \geq 0$ and $y (0) \geq 0$ then $(x (t), y (t)) \in \Omega_0$ at all positive times $t$, which establishes the
 first part of the lemma.
 This also implies that, starting from any initial condition in $\Omega_0$,
 $$ \begin{array}{rcl}
    \dot{x} \ + \ \dot{y} \ & = & \ rx \,(x - \theta) (1 - x) \ + \ ry \,(y - \theta) (1 - y) \vspace{4pt} \\
                          \ & = & \ r \ [- (x^3 + y^3) + (1 + \theta) (x^2 + y^2) - \theta (x + y)] \ < \ 0 \end{array} $$
 whenever $x + y$ is larger than some $M (\theta) > 0$.
 Therefore,
 $$ \max (x (t), y (t)) \ \leq \ x (t) + y (t) \ \leq \ M (\theta) \quad \hbox{for all $t$ large enough}. $$
 This completes the proof of Lemma \ref{p_b}. \hspace{2mm} $\blacksquare$
\end{proof}

\noindent It follows from the previous lemma that, excluding the initial condition in which both patches are initially empty,
 the population densities in both patches are simultaneously positive at any positive time.
 This implies in particular that the trivial equilibrium $E_0$ is the only boundary equilibrium.

\begin{lemma}
\label{c1}
 If $(x (0), y (0)) \in \R^2_+ \setminus \{(0,0) \}$ then $x (t) > 0$ and $y (t) > 0$ for all $t > 0$.
\end{lemma}
\begin{proof}
 By symmetry, we may assume that $x (0) > 0$ and $y (0) \geq 0$.
 We first apply Lemma \ref{p_b} to get
 $$ 0 \ \leq \ x (t), y (t) \ \leq \ M := \max (M (\theta), x (0), y (0)) \quad \hbox{at all times} \ t \geq 0 $$
 where $M (\theta)$ is as in the proof of Lemma \ref{p_b}.
 In particular,
 $$ \begin{array}{rcl}
     x' (t) \ & = & \ rx \,(x - \theta) (1 - x) \ + \ \mu \,(y - x) \ \geq \ [r (x - \theta) (1 - x) - \mu] \ x \vspace{4pt} \\
            \ & \geq & \ - \ [r \max \,(\theta, (M - \theta) (M - 1)) + \mu] \ x \ \geq \ - K \,x \end{array} $$
 for some constant $K < \infty$. Therefore,
 $$ x (t) \ \geq \ x (0) \,\exp (- K t) \ > \ 0 \quad \hbox{for all} \ t \geq 0. $$
 Finally, if $y (0) > 0$ then the same holds for $y (t)$, while if $y (0) = 0$ then
 $$ y' (0) \ = \ \mu \,x (0) \ > \ 0 $$
 which implies that $y (t) > 0$ for all $t \in (0, \ep)$ for some small $\ep > 0$.
 The fact that this holds at all times follows from the same reasoning as before based on the fact that both functions
 are bounded. \hspace{2mm} $\blacksquare$
\end{proof}

\noindent The next lemma, which also follows from Lemma \ref{p_b}, is our main tool to
 prove Theorems \ref{th:simple}-\ref{th:attractors}.
 It lists some of the invariant sets of the system.

\begin{lemma}
\label{th:invariant}
 The following sets are positive invariant for the system \eqref{d_x}-\eqref{d_y}.
 $$ \begin{array}{rcl}
    \Omega_{\theta}    & := & \{(x, y) \in \Omega_0 : x \geq \theta \ \hbox{and} \ y \geq \theta \} \vspace{4pt} \\
    \Omega_1           & := & \{(x, y) \in \Omega_0 : x \geq 1 \ \hbox{and} \ y \geq 1 \} \vspace{4pt} \\
    \Omega_{0, \theta} & := & \{(x, y) \in \Omega_0 : 0 \leq x \leq \theta \ \hbox{and} \ 0 \leq y \leq \theta \} \vspace{4pt} \\
    \Omega_{\theta, 1} & := & \{(x, y) \in \Omega_0 : \theta \leq x \leq 1 \ \hbox{and} \ \theta \leq y \leq 1 \} \vspace{4pt} \\
    \Omega_{x < y}     & := & \{(x, y) \in \Omega_0 : x < y \} \vspace{4pt} \\
    \Omega_{x > y}     & := & \{(x, y) \in \Omega_0 : x > y \} \vspace{4pt} \\
    \Omega_{x = y}     & := & \{(x, y) \in \Omega_0 : x = y \}. \end{array} $$
 Moreover, the dynamics along the invariant set $\Omega_{x = y}$ are described by
\begin{enumerate}
 \item If $x_0 = y_0 \in (0, \theta)$ then $x (t) = y (t) \ \to \ 0$ as $t \ \to \ \infty$. \vspace{4pt}
 \item If $x_0 = y_0 \in (\theta, \infty)$ then $x (t) = y (t) \ \to \ 1$ as $t \ \to \ \infty$.
\end{enumerate}
\end{lemma}
\begin{proof}
 First, we assume that the initial condition $(x_0, y_0) \in \Omega_{\theta}$ and introduce
 $$ u (t) \ = \ x (t) - \theta \quad \hbox{and} \quad v (t) \ = \ y (t) - \theta. $$
 Then, the system \eqref{d_x}-\eqref{d_y} can be rewritten as
\begin{eqnarray}
\label{d_u}
 \dot{u} \ = \ \dot{x} \ & = & \ r \,u \,(u + \theta) (1 - \theta - u) \ - \ \mu \,(v - u) \vspace{4pt} \\
\label{d_v}
 \dot{v} \ = \ \dot{y} \ & = & \ r \,v \,(v + \theta) (1 - \theta - v) \ - \ \mu \,(u - v)
\end{eqnarray}
 with initial condition $(u_0, v_0) \in \Omega_0$.
 Now, the arguments of the proof of Lemma \ref{p_b} imply that $\Omega_0$ is positive invariant
 for \eqref{d_u}-\eqref{d_v}. Moreover,
 $$ (u (t), v (t)) \in \Omega_0 \quad \hbox{if and only if} \quad (x (t), y (t)) \in \Omega_{\theta} $$
 so the set $\Omega_{\theta}$ is positive invariant for the system \eqref{d_x}-\eqref{d_y}.
 The fact that $\Omega_1$ is positive invariant follows from the same argument but applied to
 $$ u (t) \ = \ x (t) - 1 \quad \hbox{and} \quad v (t) \ = \ y (t) - 1. $$
 To prove the positive invariance of $\Omega_{0, \theta}$ we first observe that Lemma \ref{p_b} implies that any
 trajectory starting from a point in the square $\Omega_{0, \theta}$ cannot exit the square crossing its left of
 bottom side.
 Moreover, the same arguments as in the proof of Lemma \ref{p_b} imply that it cannot exit the square crossing
 its right or top sides either because of the following three properties.
\begin{enumerate}
\item The upper right corner $(\theta, \theta)$ is a fixed point of the system \eqref{d_x}-\eqref{d_y}. \vspace{4pt}
\item If $(x (t), y (t)) \in \Omega_{0, \theta}$ with $x (t) = \theta$ then
 $$ x' (t) \ = \ \mu \,(y (t) - x (t)) \ = \ \mu \,(y (t) - \theta) \ < \ 0. $$
\item If $(x (t), y (t)) \in \Omega_{0, \theta}$ with $y (t) = \theta$ then
 $$ y' (t) \ = \ \mu \,(x (t) - y (t)) \ = \ \mu \,(x (t) - \theta) \ < \ 0. $$
\end{enumerate}
 This proves that $\Omega_{0, \theta}$ is positive invariant.
 The fact that the square $\Omega_{\theta, 1}$ is also positive invariant follows from the same argument, looking at the
 derivatives along each side and using that the four corners are equilibriums.
 To prove the positive invariance of the last three sets, we introduce the new functions
 $$ u (t) \ = \ \frac{x (t) + y (t)}{2} \quad \hbox{and} \quad v (t) \ = \ \frac{x (t) - y (t)}{2}. $$
 A straightforward calculation shows that
\begin{eqnarray}
\label{du}
 \dot{u} \ & = & \ r \,u \,(u - \theta) (1 - u) \ + \ r \,v^2 \,(1 + \theta - 3u) \vspace{4pt} \\
\label{dv}
 \dot{v} \ & = & \ r \,v \,(- 3 u^2 + 2 (1 + \theta) u - v^2 - \theta - 2 \mu / r).
\end{eqnarray}
 From \eqref{dv}, we see that $v = 0$ is an invariant manifold of $v$, i.e,
 $$ v (0) = 0 \ \hbox{implies} \ v (t) = 0 \ \hbox{for all} \  t > 0, $$
 from which it follows that the set $\Omega_{x = y}$ is positive invariant for the original system \eqref{d_x}-\eqref{d_y}.
 In particular, if $v_0 = 0$ then \eqref{du} reduces to
 $$ \dot{u} \ = \ r \,u \,(u - \theta) (1 - u). $$
 Therefore, by applying Lemma \ref{l_sae}, we can conclude that
\begin{enumerate}
 \item If $u_0 \in (0, \theta)$ then $u (t) \ \to \ 0$ as $t \ \to \ \infty$, \vspace{4pt}
 \item If $u_0 \in (\theta, \infty)$ then $u (t) \ \to \ 1$ as $t \ \to \ \infty$,
\end{enumerate}
 which, in view of the definition of $u$ and $v$, and the fact that $\Omega_{x = y}$ is positive
 invariant, is equivalent to the last two statements of Lemma \ref{th:invariant}.
 Finally, for any initial condition $x_0 > y_0$, Lemma \ref{p_b} implies that $u (t)$ and $v (t)$ are both bounded uniformly
 in time so, using equation \eqref{dv} and the same argument as in the proof of Lemma \ref{c1}, we can deduce that
 $$ x (t) - y (t) \ = \ 2 \,v (t) \ \geq \ 2 \,v_0 \,\exp (- K t) \ = \ (x_0 - y_0) \,\exp (- K t) \ > \ 0 $$
 for all $t \geq 0$ and some constant $K < \infty$.
 This proves that $\Omega_{x > y}$ is positive invariant.
 By symmetry, the same holds for the set $\Omega_{x < y}$. \hspace{2mm} $\blacksquare$
\end{proof}

\noindent With Lemmas \ref{p_b}-\ref{th:invariant} in hands, we are now ready to prove the main results for the deterministic
 two-patch model described by the system \eqref{d_x}-\eqref{d_y}.


\subsection*{Proof of Theorem \ref{th:simple}}

\indent By Poincar\'e-Bendixson Theorem, the omega limit set of the system \eqref{d_x}-\eqref{d_y} is either a fixed point or a limit cycle.
 If the inequality \eqref{simple_values} holds, we can use Dulac's criterion to exclude the existence of a limit cycle.
 Let $c \in [0, 3)$ and define the scalar function $p_c (x, y) = (xy)^{-c}$ on $\R^2_+$. Then,
 $$ \begin{array}{l}
    \displaystyle \frac{\partial}{\partial x} \ [(r x \,(x - \theta) (1 - x) + \mu \,(y - x)) \,p(x, y)] \\ \hspace{120pt} + \
    \displaystyle \frac{\partial}{\partial y} \ [(r y \,(y - \theta) (1 - y) + \mu \,(x - y)) \,p(x, y)] \vspace{8pt} \\ = \
    \displaystyle (xy)^{-c} \ [r (c - 3) (x^2 + y^2) + r (2 - c) (1 + \theta) (x + y) \\ \vspace{-5pt} \\ \hspace{120pt} + \
     2 r \theta (c - 1) - 2 \mu + c \mu (2 - x y^{-1} - y x^{-1})] \vspace{8pt} \\ \leq \
    \displaystyle (xy)^{-c} \ \bigg[r (c - 3) \bigg(x + \frac{(2 - c)(1 + \theta)}{2 (c - 3)} \bigg)^2 - \frac{r (2 - c)^2 (1 + \theta)^2}{4 (c - 3)} \vspace{8pt} \\ \hspace{5pt} + \
    \displaystyle                   r (c - 3) \bigg(y + \frac{(2 - c)(1 + \theta)}{2 (c - 3)} \bigg)^2 - \frac{r (2 - c)^2 (1 + \theta)^2}{4 (c - 3)}
    \displaystyle + 2 r \theta (c - 1) - 2 \mu \bigg]. \end{array}$$
 In particular, if \eqref{simple_values} holds then the equation above is strictly negative for any $(x, y) \in \Omega \setminus E_0$.
 Therefore, by Dulac's criterion, the system has no limit cycle, i.e., any trajectory of \eqref{d_x}-\eqref{d_y} starting with
 a nonnegative initial condition converges to a fixed point. \hspace{2mm} $\blacksquare$


\subsection*{Proof of Theorem \ref{th:basin}}

\indent Parts 1 and 2 of Theorem \ref{th:basin} about the local stability of the three symmetric equilibriums follow from the analysis
 of the Jacobian matrices.
 For each of the three equilibriums, we have
\begin{description}
\item {$E_0$ --} The Jacobian matrix associated with this equilibrium is
\begin{equation}
\label{J00}
 J_0 \ = \ \left(\begin{array}{cc} - r \theta - \mu & \mu \\ \noalign{\medskip} \mu & - r \theta - \mu \end{array} \right)
\end{equation}
 with eigenvalues $\lambda_1 = - r \theta$ and $\lambda_2 = - r \theta - 2 \mu$ associated with $(1, 1)$ and $(-1, 1)$ as their eigenvectors,
 respectively.
 We can easily conclude that the trivial boundary equilibrium $E_0$ is locally stable since both eigenvalues of \eqref{J00} are negative. \vspace{4pt}
\item {$E_{\theta}$ --} The Jacobian matrix associated with this equilibrium is
\begin{equation}
\label{Jtt}
 J_{\theta} \ = \ \left(\begin{array}{cc}  r \theta(1-\theta) - \mu & \mu \\ \noalign{\medskip} \mu & r \theta(1-\theta) \end{array} \right)
\end{equation}
 with eigenvalues $\lambda_1 =  r \theta (1 - \theta)$ and $\lambda_2 = r \theta (1 - \theta) - 2 \mu$ associated with $(1, 1)$ and $(-1, 1)$ as their
 eigenvectors, respectively.
 We can easily conclude that the equilibrium $E_{\theta}$ is always unstable on the invariant set $\Omega_{x = y}$.
 Moreover, if $2 \mu > r \theta (1 - \theta)$ then $E_{\theta}$ is a saddle, while if  $2 \mu < r \theta (1 - \theta)$ then $E_{\theta}$ is a source. \vspace{4pt}
 \item {$E_1$ --} The Jacobian matrix associated with this equilibrium is
\begin{equation}
\label{J11}
 J_1 \ = \ \left(\begin{array}{cc} - r (1 - \theta) - \mu & \mu \\ \noalign{\medskip} \mu & - r (1 - \theta) - \mu \end{array} \right)
\end{equation}
 with two negative eigenvalues $\lambda_1 = - r (1 - \theta) < 0$ and $\lambda_2 = - r (1 - \theta) - 2 \mu < 0$ since $\theta < 1$.
 Therefore, the equilibrium $E_1$ is also locally stable.
\end{description}
 To prove the third part of the theorem, we first define the function $u (t) = x (t) + y (t)$. Then
 $$ \dot{u} \ = \ \dot{x} \ + \ \dot{y} \ = \ r x \,(x - \theta) (1 - x) \ + \ r y \,(y - \theta) (1 - y). $$
 To prove that
\begin{equation}
\label{eq:basin1}
 \Omega_{0, \theta} \setminus \{(\theta, \theta) \} \ \subset \ B_0
\end{equation}
 we first assume that
 $$ (x_0, y_0) \in \Omega_{0, \theta} \setminus \{E_0, E_{\theta} \}. $$
 Since the set $\Omega_{0, \theta}$ is positive invariant, we have $u' (t) \leq 0$ for all $t \geq 0$.
 Using in addition that
 $$ u' (t) = 0 \quad \hbox{if and only if} \quad u (t) = 0, $$
 we can conclude that $u (t)$ converges to zero.
 Recalling the definition of $u$ and invoking again the positive invariance of $\Omega_{0, \theta}$, we can deduce that $x (t)$
 and $y (t)$ converge to zero so \eqref{eq:basin1} holds.
 To prove that
\begin{equation}
\label{eq:basin2}
 \Omega_{\theta} \setminus \{(\theta, \theta) \} \ \subset \ B_1
\end{equation}
 we now assume that
 $$ (x_0, y_0) \in \Omega_{\theta} \setminus \{E_{\theta}, E_1 \}. $$
 Then, we have the following alternative.
\begin{enumerate}
 \item $(x_0, y_0) \in \Omega_{\theta, 1} \setminus \{E_{\theta}, E_1 \}$.
  Since $\Omega_{\theta, 1}$ is positive invariant, we may use the same argument as before to see that the derivative of $u$ is
  nonnegative and the system converges to the equilibrium point $E_1$. \vspace{4pt}
 \item $(x_0, y_0) \in \Omega_1 \setminus \{E_1 \}$.
  Repeating again the same argument but with the positive invariant set $\Omega_1$ implies that the system converges to $E_1$. \vspace{4pt}
 \item $(x_0, y_0) \in \Omega_{\theta} \setminus (\Omega_{\theta, 1} \cup \Omega_1)$.
  We may assume that $x_0 < y_0$ without loss of generality since the system is symmetric.
  Then, using the positive invariance of the set $\Omega_{x < y}$ we have $x (t) < y (t)$ for all $t \geq 0$ so
  $$ \begin{array}{rcll}
     \dot{x} \ & = & \ r x \,(x - \theta) (1 - x) + \mu \,(y - x) \ > \ 0 & \quad \hbox{if} \ x \leq 1 \vspace{4pt} \\
     \dot{y} \ & = & \ r y \,(y - \theta) (1 - y) + \mu \,(x - y) \ < \ 0 & \quad \hbox{if} \ y \geq 1. \end{array} $$
  This indicates that the trajectory starting at $(x_0, y_0)$ can only exit the infinite rectangle $[\theta, 1] \times [1, \infty)$ by crossing
  its bottom or right side.
  Therefore, we have the following three possibilities. \vspace{4pt}
 \begin{enumerate}
  \item[a.] No exit: $(x (t), y (t)) \notin \Omega_{\theta, 1} \cup \Omega_1$ for all $t \geq 0$.
   In this case, the sign of the derivatives implies convergence to $E_1$. \vspace{4pt}
  \item[b.] Bottom side: $(x (t), y (t)) \in \Omega_{\theta, 1}$ for some time $t \geq 0$.
   In this case, point 1 above implies convergence to the equilibrium point $E_1$. \vspace{4pt}
  \item[b.] Right side: $(x (t), y (t)) \in \Omega_1$ for some time $t \geq 0$.
   In this case, point 2 above implies convergence to the equilibrium point $E_1$.
 \end{enumerate}
\end{enumerate}
 Combining 1-3 above implies \eqref{eq:basin2}.
 Now assume that $\theta < 1/2$.
 Defining
 $$ u (t) \ = \ \frac{x (t) + y (t)}{2} \quad \hbox{and} \quad v (t) \ = \ \frac{x (t) - y (t)}{2} $$
 recall that the system \eqref{d_x}-\eqref{d_y} can be rewritten as
 $$ \begin{array}{rcl}
    \dot{u} \ & = & \ r \,u \,(u - \theta) (1 - u) \ + \ r \,v^2 \,(1 + \theta - 3u) \vspace{4pt} \\
    \dot{v} \ & = & \ r \,v \,(- 3 u^2 + 2 (1 + \theta) u - v^2 - \theta - 2 \mu / r). \end{array} $$
 To prove that
\begin{equation}
\label{eq:basin3}
 B_0 \ \subset \ \{(x, y) \in \Omega_0 : x + y < 2 \theta \}
\end{equation}
 it suffices to prove that
\begin{equation}
\label{eq:basin}
 x_0 + y_0 \ \geq \ 2 \theta \quad \hbox{implies} \quad x (t) + y (t) \ \geq \ 2 \theta \quad \hbox{for all} \ t \geq 0.
\end{equation}
 Assume by contradiction that \eqref{eq:basin} is not satisfied.
 Then, there exists an initial condition with $x_0 + y_0 \geq 2 \theta$ and a time $T > 0$ such that
 $$ x (0) + y (0) \ = \ x_0 + y_0 \ \geq \ 2 \theta \quad \hbox{and} \quad x (T) + y (T) \ < \ 2 \theta. $$
 By continuity of the trajectories, the intermediate value theorem implies the existence of a time $t < T$ such that $u (t) = 2 \theta$
 therefore
 $$ S \ := \ \sup \,\{t < T : u (t) = 2 \theta \} $$
 is well defined and $u (t) < 2 \theta$ for all $t \in (S, T]$.
 To prove that this leads to a contradiction, we consider the following two cases.
\begin{enumerate}
 \item If $x_0 \neq y_0$, the invariance of $\Omega_{x < y}$ and $\Omega_{x > y}$ implies that $x (t) \neq y (t)$ at any time,
  from which it follows that
  $$ u' (S) \ = \ r \,v^2 (S) \,(1 + \theta - 3 \theta) \ = \ (r / 4) \,(x (S) - y (S))^2 (1 - 2 \theta) \ > \ 0. $$
 In particular, there exists $\ep > 0$ such that $u (t) > 2 \theta$ for all $t \in (S, S + \ep)$, which contradicts the
 existence of time $S$. \vspace{4pt}
\item If $x_0 = y_0$ the result directly follows from the fact that $\Omega_{\theta} \cap \Omega_{x = y}$ is
 positive~invariant, as it is the intersection of two invariant sets.
\end{enumerate}
 Combining 1 and 2 above yields \eqref{eq:basin3}.
 The proof of the last inclusion in Theorem \ref{th:basin} follows from similar arguments. \hspace{2mm} $\blacksquare$


\subsection*{Proof of Theorem \ref{th:dispersal}}

\indent Define as previously the functions
 $$ u (t) \ = \ \frac{x (t) + y (t)}{2} \quad \hbox{and} \quad v (t) \ = \ \frac{x (t) - y (t)}{2}. $$
 Using \eqref{dv} above, we obtain
 $$ \dot{v} \ = \ r \,v \,\left(-3 \,\bigg(u - \frac{1 + \theta}{3} \bigg)^2 - \frac{2}{r} \,\bigg(\mu - \frac{r (\theta^2 - \theta + 1)}{6} \bigg) - v^2 \right). $$
 Assume first that $x_0 > y_0$.
 Using the fact that the set $\Omega_{x > y}$ is positive invariant by Lemma \ref{th:invariant}, we deduce that $x (t) > y (t)$ at
 all positive times $t$.
 In particular, recalling the definition of $v$, and using the expression of the derivative $\dot{v}$ above and the fact
 that \eqref{eq:dispersal} holds, we obtain that
 $$ v (0) > 0 \quad \hbox{implies} \quad v (t) > 0 \ \hbox{and} \ v' (t) < 0 \ \hbox{for all} \ t \geq 0. $$
 Since $\dot{v} = 0$ if and only if $v = 0$, we deduce that $v (t)$ converges to 0.
 By symmetry, the same can be proved of the system starting with any initial conditions such that $x_0 < y_0$.
 Since $\Omega_{x = y}$ is positive invariant, we have the same conclusion when the initial condition satisfies $x_0 = y_0$, which can
 also be seen from the expression of the derivative $\dot{v}$.
 Therefore, if \eqref{eq:dispersal} holds then
 $$ \hbox{for all} \ (x_0, y_0) \in \R^2_+, \quad \lim_{t \to \infty} v (t) = 0, $$
 so for any  $(x_0, y_0) \in \R^2_+$ and any $\epsilon > 0$, there exists $k > 0$ such that
 $$ |x (t) - y(t)| \ < \ \epsilon \qquad \hbox{for all} \ t > k. $$
 It follows that any trajectory of the system converges to one of the symmetric equilibriums $E_0$, $E_{\theta}$ or $E_1$.
 Now, observing that
 $$ \frac{r (\theta^2 - \theta + 1)}{6} \ - \ \frac{r \theta (1 - \theta)}{2} \ = \ \frac{r (2 \theta - 1)^2}{6} \ \geq \ 0, $$
 we obtain that
 $$ \mu \ > \ \frac{r (\theta^2 - \theta + 1)}{6} \ \geq \ \frac{r \theta (1 - \theta)}{2}. $$
 In view of the expression of the Jacobian matrix \eqref{Jtt}, this implies that the equilibrium $(\theta, \theta)$ is a saddle
 with unstable manifold
 $$ \Omega_{x = y} \setminus \{E_0, E_{\theta}, E_1 \}. $$
 In particular, it follows from Hartman-Grobman Theorem and the second part of Lemma \ref{th:invariant} that there are only two
 attractors: $E_0$ and $E_1$.
 Hence, the system starting from any initial condition not belonging to the manifold $S_{\theta}$ converges to
 either $E_0$ or $E_1$.
 To conclude the proof, it suffices to observe that, by \eqref{eq:basin} in the proof of Theorem \ref{th:basin}, if $\theta < 1/2$
 then the set
 $$ \{(x, y) \in \Omega_0 \setminus S_{\theta} : x + y \geq 2 \theta\} $$
 is positive invariant so the system starting from any initial condition in this set converges to $E_1$, the only attractor in
 this invariant set.
 Similarly, the last statement follows from the fact that, if $\theta > 1/2$ then the set
 $$ \{(x, y) \in \Omega \setminus S_{\theta} : x + y \leq 2 \theta\} $$
 is positive invariant and contains only one attractor: $E_0$. \hspace{2mm} $\blacksquare$


\subsection*{Proof of Theorem \ref{th:attractors}}

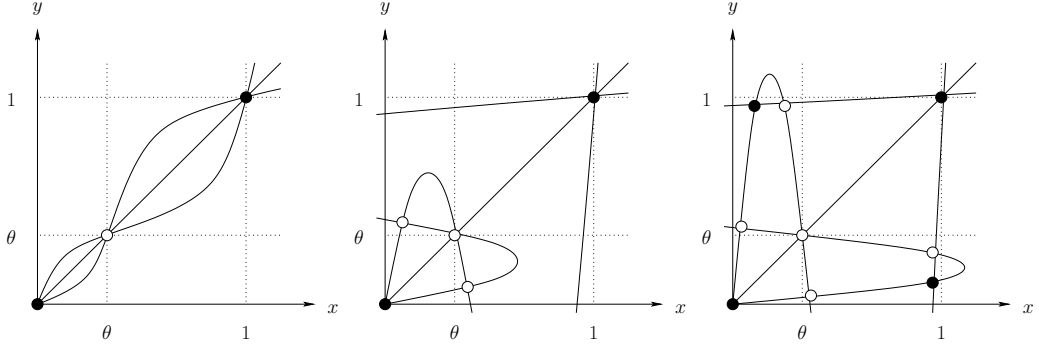
\begin{figure}[t]
\centering
\scalebox{0.36}{\input{nullclines.pstex_t}}
\caption{Schematic presentations of nullclines of the system with black dots representing locally stable equilibriums.}
\label{fig:nullclines}
\end{figure}

\indent We first prove that all the equilibriums of the system \eqref{d_x}-\eqref{d_y} belong to the unit square $[0, 1] \times [0, 1]$.
 Since the system is symmetric and, by Lemma \ref{th:invariant}, the omega limit set of any initial condition
 $(x_0, y_0) \in \Omega_{0, \theta} \cup \Omega_\theta$ belongs to the unit square (either $E_0$, $E_{\theta}$ or $E_1$),
 it suffices to focus on the case
 $$ (x_0, y_0) \in \{(x, y) \in \Omega_0 : 0 \leq x \leq \theta \leq y \}. $$
 Since $x_0 = y_0$ only happens when starting from $(\theta, \theta)$, to avoid trivialities, we shall assume in addition that $x_0 < y_0$.
 Then, applying Lemma \ref{th:invariant}, we obtain that $x (t) < y (t)$ for all $t \geq 0$ which, together with \eqref{d_y}, implies that
\begin{equation}
\label{d_y2}
 \dot{y} \ = \ r y \,(y - \theta) (1 - y) + \mu \,(x - y) \ < \ 0 \quad \hbox{if} \ y \geq 1.
\end{equation}
 Excluding the trivial case when the initial condition belongs to the stable manifold of $(\theta, \theta)$, in which case its omega
 limit set reduces to $(\theta, \theta)$, we have the following alternative.
\begin{enumerate}
\item If $(x (t), y (t)) \in \Omega_{\theta}$ for some $t \geq 0$ then, by the second part of Theorem \ref{th:basin}, the omega limit
 set of the initial condition is $E_1$. \vspace{4pt}
\item If $(x (t), y (t)) \notin \Omega_{\theta}$ for all $t \geq 0$ then \eqref{d_y2} and the fact that $\Omega_{x < y}$ is positive
 invariant imply that $x (T) < y (T) \leq 1$ for some $T > 0$.
 Using as previously the continuity of the trajectories and the fact that
 $$ \dot{y} \ = \ \mu \,(x - y) \ < \ 0 \quad \hbox{if} \ y = 1 $$
 allows to invoke the intermediate value theorem and prove by contradiction that $x (t) < y (t) \leq 1$ at any time $t \geq T$.
\end{enumerate}
 This establishes the first part of Theorem \ref{th:attractors}.
 The second part follows directly from the proof of Theorem \ref{th:dispersal}.
 To prove the third part, we first observe that the equation of the nullcline $y = f (x)$ can be rewritten as
 $$ y \ = \ \frac{r x}{\mu} \left[\bigg(x - \frac{1 + \theta}{2} \bigg)^2 - \frac{(1 - \theta)^2}{4} + \frac{\mu}{r} \right]. $$
 In particular, if $(1 - \theta)^2 > 4 \mu / r$ then the nullcline intersects the x-axis at the three points with coordinates
 $(0, 0)$, $(x_1, 0)$ and $(x_2, 0)$ where
 $$ \begin{array}{rcl}
     x_1 & = & \displaystyle \frac{1 + \theta}{2} - \sqrt{\frac{(1 - \theta)^2}{4} - \frac{\mu}{r}} \vspace{10pt} \\
     x_2 & = & \displaystyle \frac{1 + \theta}{2} + \sqrt{\frac{(1 - \theta)^2}{4} - \frac{\mu}{r}}. \end{array} $$
 Finally, a phase-plane analysis based on Figure \ref{fig:nullclines} shows that
\begin{enumerate}
 \item If $x_1 < M < x_2$ then the system has five fixed points with only two locally stable equilibriums: $E_0$ and $E_1$. \vspace{4pt}
 \item If $M \geq 1$ then the system achieves maximum number of equilibriums which is nine, with only four locally stable equilibriums. \vspace{4pt}
 \item If the system has less than nine equilibriums, then it has only two local stable equilibriums: $E_0$ and $E_1$.
\end{enumerate}
 Note also that $M$ can be computed explicitly: $M = f (x^*)$ where
 $$ x^* \ = \ \frac{1}{3} \ \left((1 + \theta) - \sqrt{\theta^2 - \theta + 1 - 3 \mu / r} \right) $$
 is the smallest root of the polynomial $h' (x)$. \hspace{2mm} $\blacksquare$


\subsection*{Proof of Theorem \ref{fixation}}

\indent The first step is to prove that the set of upper configurations is closed under the dynamics.
 We observe that, condition on the event that the configuration is an upper configuration, only expansions and
 migrations can occur.
 Furthermore, migration events can only result in an increase of the lowest density and a decrease of the highest density,
 i.e., if a migration event occurs at time $t$ and the configuration at time $t - \Delta t$ is an upper configuration then
 $$ \begin{array}{l}
    \min (X_{t - \Delta t}, Y_{t - \Delta t}) \ \leq \ \min (X_t, Y_t) \vspace{4pt} \\ \hspace{40pt} \leq \
    \max (X_t, Y_t) \ \leq \ \max (X_{t - \Delta t}, Y_{t - \Delta t}). \end{array} $$
 It follows that the set of upper configurations (and similarly the set of lower configurations) is closed under the
 dynamics, i.e.,
 $$ \begin{array}{l}
  P \,((X_t, Y_t) \in \Omega^+ \ | \ (X_0, Y_0) \in \Omega^+) \vspace{4pt} \\ \hspace{40pt} = \
  P \,((X_t, Y_t) \in \Omega^- \ | \ (X_0, Y_0) \in \Omega^-) \ = \ 1 \end{array} $$
 for all times $t$.
 Since, starting from an upper configuration, the system jumps to $(1, 1)$ whenever two
 expansion events at $X$ and $Y$ occur consecutively (they are not separated by a migration event), we deduce that
 the stopping time $\tau^+$ is almost surely finite.
 The same holds for the stopping time $\tau^-$ when starting from a lower configuration. Hence,
 $$ \begin{array}{l}
  P \,(\tau = \tau^+ < \infty \ | \ (X_0, Y_0) \in \Omega^+) \vspace{4pt} \\ \hspace{40pt} = \
  P \,(\tau = \tau^- < \infty \ | \ (X_0, Y_0) \in \Omega^-) \ = \ 1. \end{array} $$
 To compute the expected value of the time to fixation, we now construct the stochastic process graphically from a collection
 of Poisson processes, relying on an idea of Harris (1972).
 Two Poisson processes, each with parameter $r$, are attached to each of the patches $X$ and $Y$, and an additional Poisson
 process with parameter one is attached to the edge connecting the patches.
 All three processes are independent. Let
 $$ \Gamma_X \ = \ \{T_n^X : n \geq 1 \}, \quad
    \Gamma_Y \ = \ \{T_n^Y : n \geq 1 \}, \quad
    \Gamma_e \ = \ \{T_n^e : n \geq 1 \} $$
 denote these Poisson processes.
 At any time of the process $\Gamma_X$ the population size at patch $X$ jumps to either 0 or 1 depending on whether it is
 smaller or larger than $\theta$ by this time, respectively.
 The evolution at patch $Y$ is defined similarly but using the Poisson process $\Gamma_Y$.
 At each time in $\Gamma_e$, a fraction $\mu$ of the population at each patch is displaced to the other patch.
 To compute the expected value, we let $t \geq 0$ and introduce the stopping times
 $$ T_Z \ = \ \min \,\{\Gamma_Z \cap (t, \infty) \} \ \ \hbox{for} \ Z = X, Y, e. $$
 Then, $P \,(\max (T_X, T_Y) < T_e)$ is the probability that two consecutive migration events are separated by at least one
 extinction-expansion event at patch $X$ and one extinction-expansion event at patch $Y$.
 To compute this probability, we first observe that $T_X$ and $T_Y$ are independent exponentially distributed random variables
 with parameter $r$, from which it follows that
 $$ P \,(\max (T_X, T_Y) < u) \ = \ P \,(T_X < u, \,T_Y < u) \ = \ (1 - \exp (-ru))^2. $$
 Since $T_e$ is exponentially distributed with parameter 1,
 $$ \begin{array}{rcl}
  P \,(\max (T_X, T_Y) < T_e) & = &
    \displaystyle \int_0^{\infty} \int_u^{\infty} \ e^{-v} \ \frac{d}{du} \Big((1 - \exp (-ru))^2 \Big) \,dv \,du \vspace{8pt} \\ & = &
    \displaystyle \int_0^{\infty} e^{-u} \ \frac{d}{du} \Big((1 - \exp (-ru))^2 \Big) \,du \vspace{8pt} \\ & = &
    \displaystyle \frac{2r}{r + 1} - \frac{2r}{2r + 1} \ = \ \frac{2r^2}{(r + 1)(2r + 1)} \ := \ p_s. \end{array} $$
 Hence, the last time a migration event occurs before fixation is equal in distribution to $T_{J - 1}^e$ where the random variable $J$
 is a geometrically distributed with parameter $p_s$ from which we deduce that
 $$ \begin{array}{l}
    \E \,[\,\tau^+ \ | \ \theta < X_0, Y_0 < 1] \ = \ \E \,[T_e] \times \E \,[J - 1] \ + \ \E \,[\max (T_X, T_Y)] \vspace{8pt} \\
    \hspace{50pt} = \ \displaystyle \frac{(r + 1)(2r + 1)}{2r^2} \ - \ 1 \ + \ \int_0^{\infty} P \,(\max (T_X, T_Y) > u) \,du \vspace{8pt} \\
    \hspace{50pt} = \ \displaystyle \frac{3r + 1}{2r^2} \ + \ \int_0^{\infty} 1 - (1 - \exp (-ru))^2 \,du \ = \
    \frac{6r + 1}{2r^2}. \end{array} $$
 The same holds for the stopping time $\tau^-$ when starting the process from a lower configuration.
 This completes the proof of Theorem \ref{fixation}. \hspace{2mm} $\blacksquare$


\subsection*{Proof of Theorem \ref{metastability}}

\indent We first prove that $P \,(T < \infty) = P \,(\tau < \infty) = 1$.
 Let $\epsilon > 0$ small.
 Then, for almost all realizations of the process, there exists an increasing sequence of random times
 $T_1 < \cdots < T_i < \cdots$ such that
 $$ \lim_{i \to \infty} T_i \ = \ \infty \quad \hbox{and} \quad |X_{T_i} + Y_{T_i} - 2 \theta| \ > \ \epsilon \quad \hbox{for all} \ i \geq 1. $$
 Moreover, there exists $K < \infty$ that does not depend on $i$ such that, if after $T_i$ a sequence of $K$ migration events
 occur before any expansion or extinction events then the system hits either an upper configuration or a lower configuration.
 Since $K$ is finite, such an event has a strictly positive probability, so the Borel-Cantelli Lemma implies that the process hits either
 an upper configuration or a lower configuration after a random time which is almost surely finite: $P \,(T < \infty) = 1$.
 Theorem \ref{fixation} then implies that
 $$ \begin{array}{rcl}
  P \,(\tau < \infty) \ & = & \
  P \,(\tau^+ < \infty) \ + \ P \,(\tau^- < \infty) \vspace{4pt} \\ & \geq & \
  P \,(\tau^+ < \infty \ | \ (X_0, Y_0) \in \Omega^+) \,P \,(T^+ < \infty) \vspace{4pt} \\ && \hspace{20pt} + \
  P \,(\tau^- < \infty \ | \ (X_0, Y_0) \in \Omega^-) \,P \,(T^- < \infty) \vspace{4pt} \\ & = & \
  P \,(T^+ < \infty) \ + \ P \,(T^- < \infty) \ = \ P \,(T < \infty) \ = \ 1. \end{array} $$
 To estimate the expected value of $T$, we observe that the transition rates of the process indicate that if at time $t$ exactly $n$
 migration events but neither expansion nor extinction events have occurred then
 $$ (X_t, Y_t) = f^n (0, 1) \quad \hbox{where} \quad f (a, b) = (1 - \mu) \,(a, b) + \mu \,(b, a), $$
 so that $X_t \leq u_n$ and $Y_t \geq v_n$ where $u_n$ and $v_n$ are defined recursively by
 $$ \begin{array}{rcl}
     u_{n + 1} & = & (1 - \mu) \,u_n + \mu \quad \hbox{with} \quad u_0 = 0, \vspace{4pt} \\
     v_{n + 1} & = & (1 - \mu) \,v_n \quad \hbox{with} \quad v_0 = 1. \end{array} $$
 A straightforward calculation shows that
 $$ u_n \ = \ \sum_{k = 0}^{n - 1} \ (u_{k + 1} - u_k) \ = \ \sum_{k = 0}^{n - 1} \ (1 - \mu)^k \,(u_1 - u_0) \ = \ 1 - (1 - \mu)^n $$
 and $v_n = 1 - u_n = (1 - \mu)^n$, therefore
 $$ \begin{array}{rcl}
     u_n > \theta & \quad \hbox{if and only if} \quad & n > n_1 := \lfloor \ln (1 - \theta) / \ln (1 - \mu) \rfloor \vspace{4pt} \\
     v_n < \theta & \quad \hbox{if and only if} \quad & n > n_2 := \lfloor \ln (\theta) / \ln (1 - \mu) \rfloor \end{array} $$
 where $\lfloor \cdot \rfloor$ is for the integer part.
 Now, let $\{(U_t, V_t) \}_t$ be the Markov process with state space
 $$ E \ = \ \{(u_i, v_j) : i, j \geq 0 \} $$
 and transition rates
 $$ \begin{array}{rcl}
  P \,((U_{t + \Delta t}, V_{t + \Delta t}) = (0, v_j) \ | \ (U_t, V_t) = (u_i, v_j)) & = & r \Delta t + o (\Delta t) \vspace{4pt} \\
  P \,((U_{t + \Delta t}, V_{t + \Delta t}) = (u_i, 1) \ | \ (U_t, V_t) = (u_i, v_j)) & = & r h + o (h) \vspace{4pt} \\
  P \,((U_{t + \Delta t}, V_{t + \Delta t}) = (u_{i + 1}, v_{j + 1}) \ | \ (U_t, V_t) = (u_i, v_j)) & = & \Delta t + o (\Delta t) \end{array} $$
 and starting at $(U_0, V_0) = (0, 1)$.
 We call $W$-, $N$-, and $SE$-jumps, the jumps described by the three transition rates above, respectively, and refer the
 reader to the left-hand side of Figure \ref{Fig: grid} for an illustration of the process.
\begin{figure}[t]
\centering
\scalebox{0.55}{\input{grid.pstex_t}}
\caption{Schematic representation of the stochastic process $(U_t, V_t)$.}
\label{Fig: grid}
\end{figure}
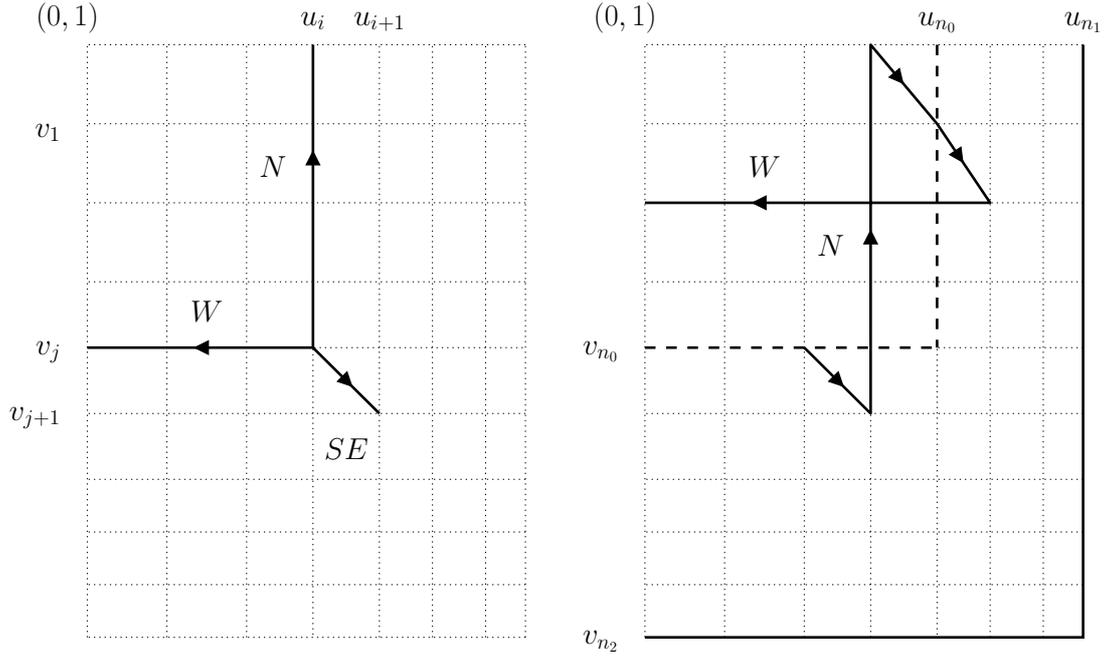
 By construction of the sequences $(u_n)_n$ and $(v_n)_n$, we have
 $$ \begin{array}{rcl}
  P \,(X_t \geq a \ | \ T > t) \ & \leq & \ P \,(U_t \geq a) \vspace{4pt} \\
  P \,(Y_t \geq a \ | \ T > t) \ & \geq & \ P \,(V_t \geq a) \end{array} $$
 for all $a \in [0, 1]$, i.e., before the process hits an upper or a lower configuration, $X_t$ is stochastically smaller than $U_t$ while
 $Y_t$ is stochastically larger than $V_t$.
 This implies that $\E \,[T] \geq \E \,[T^*]$ where
 $$ T^* \ = \ \inf \,\{t \geq 0 : U_t > u_{n_1} \ \hbox{or} \ V_t < v_{n_2} \}. $$
 $$ E_1 = \{(u_i, v_j) : i, j \leq n_0 \} \quad \hbox{and} \quad
    E_2 = \{(u_i, v_j) : i \leq n_1 \ \hbox{and} \ j \leq n_2 \}. $$
 Then, $T^*$ is the first time $(U_t, V_t)$ exits the set $E_2$, i.e.,
 $$ T^* \ = \ \inf \,\{t \geq 0 : (U_t, V_t) \notin E_2 \}. $$
 So, to bound $\E \,[T^*]$ from below, it suffices to prove that $(U_t, V_t) \in E_2$ for an arbitrarily long time.
 The idea is to prove that, when starting from the smaller rectangle $E_1$, the process stays in $E_2$ and comes back to $E_1$
 after $n_0$ jumps with probability close to 1.
 Using in addition the Markov property, we obtain that the number of jumps required to exit $E_2$ is stochastically
 larger than $n_0$ times a geometric random variable with small success probability.
 To make this argument precise, we let $(\U_n, \V_n)$ denote the embedded discrete-time Markov chain associated
 with the process $(U_t, V_t)$.
 To count the number of steps needed to exit the rectangle $E_2$, we define a sequence of Bernoulli random variables
 $\{Z_k : k \geq 1 \}$ associated to $(\U_n, \V_n)$ by setting
 $$ \begin{array}{rcll}
     Z_k & \ = \ & 0 & \quad \hbox{if there is at least one $N$-jump and one $W$-jump} \\
         &&& \hspace{75pt} \hbox{between time $(k - 1) n_0 + 1$ and time $k n_0$} \vspace{4pt} \\
         & \ = \ & 1 & \quad \hbox{if there is no $N$-jump or no $W$-jump} \\
         &&& \hspace{75pt} \hbox{between time $(k - 1) n_0 + 1$ and time $k n_0$} \end{array} $$
 Since $(\U_n, \V_n)$ is a discrete-time Markov chain, the random variables $Z_k$ are independent Bernoulli random
 variables, and a straightforward calculation shows that the success probability is given by
 $$ P \,(Z_k = 1) \ \leq \ 2 \ \bigg(1 - \frac{r}{1 + 2r} \bigg)^{n_0} \ = \ 2 \ \bigg(\frac{1 + r}{1 + 2r} \bigg)^{n_0}. $$
 Moreover, since $n_0 = (1 / 2) \,\min (n_1, n_2)$, we have that
 $$ \begin{array}{l}
   (\U_{k n_0}, \V_{k n_0}) \in E_1 \ \hbox{and} \ Z_{k + 1} = 0 \quad \hbox{implies that} \vspace{4pt} \\ \hspace{0pt}
   (\U_n, \V_n) \in E_2 \ \hbox{for all} \ k n_0 \leq n \leq (k + 1) n_0 \ \hbox{and} \
   (\U_{(k + 1) n_0}, \V_{(k + 1) n_0}) \in E_1. \end{array} $$
 See the right-hand side of Figure \ref{Fig: grid}.
 This indicates that
 $$ Z_1 = Z_2 = \cdots = Z_k = 0 \quad \Longrightarrow \quad (\U_n, \V_n) \in E_2 \ \ \hbox{for all} \ n \leq k n_0. $$
 Finally, using that $(U_t, V_t)$ jumps at rate $1 + 2r$ and that $\inf \,\{k : Z_k = 1 \}$ is stochastically larger than a
 geometric random variable $Z$ with success probability $P \,(Z_k = 1)$ we can conclude that
 $$ \E \,[T] \ \geq \ \E \,[T^*] \ \geq \ \frac{n_0}{1 + 2r} \ \E \,[Z] \ = \
    \frac{n_0}{2 + 4r} \ \bigg(\frac{1 + 2r}{1 + r} \bigg)^{n_0}. $$
 This completes the proof. \hspace{2mm} $\blacksquare$


\subsection*{Proof of Theorem \ref{metastable}}

\indent First, we observe that the process $U_t$ introduced in the proof of Theorem \ref{metastability} is stochastically
 larger than $\bar X_t$ so to prove the first inequality it suffices to establish its analog for the expected value $\E_{\pi} (U_t)$
 where $\pi$ is the stationary distribution of the stochastic process $U_t$.
 Note that the infinitesimal matrix of the Markov process $U_t$ expressed in the basis $(u_0, u_1, u_2, \ldots)$ is given by
 $$ Q \ = \ \left(\begin{array}{ccccc}
  - 1 &      1      & 0         & 0         & \cdots \vspace{3pt} \\
    r &   - (r + 1) & 1         & 0         & \cdots \vspace{3pt} \\
    r &      0      & - (r + 1) & 1         &        \vspace{-3pt} \\
    r &      0      & 0         & \ddots    & \ddots \\
   \vdots & \vdots  & \vdots    &           & \ddots \end{array} \right) $$
 By solving $\pi \cdot Q = 0$, we find that
 $$ \pi \ = \ r \ \bigg(\frac{1}{r + 1}, \bigg(\frac{1}{r + 1} \bigg)^2, \bigg(\frac{1}{r + 1} \bigg)^3, \cdots, \bigg(\frac{1}{r + 1} \bigg)^n, \cdots \bigg). $$
 This implies that
 $$ \begin{array}{l}
    \E_{\nu} \,(\bar X_t) \ \leq \ \E_{\pi} \,(U_t) \ = \
    \displaystyle r \ \sum_{n = 0}^{\infty} \ u_n \,\bigg(\frac{1}{r + 1} \bigg)^{n + 1} \vspace{8pt} \\ \hspace{20pt} = \
    \displaystyle r \ \sum_{n = 0}^{\infty} \ (1 - (1 - \mu)^n) \,\bigg(\frac{1}{r + 1} \bigg)^{n + 1} \vspace{8pt} \\ \hspace{20pt} = \
    \displaystyle \frac{r}{r + 1} \ \sum_{n = 0}^{\infty} \ \bigg(\frac{1}{r + 1} \bigg)^n \ - \
    \displaystyle \frac{r}{r + 1} \ \sum_{n = 0}^{\infty} \ \bigg(\frac{1 - \mu}{r + 1} \bigg)^n \ = \
    \displaystyle 1 - \frac{r}{r + \mu}. \end{array} $$
 The proof of the second inequality is similar. \hspace{2mm} $\blacksquare$


\subsection*{Proof of Theorem \ref{probabilities}}

\indent We first observe that the processes $(X_t, Y_t)$ and $(\bar X_t, \bar Y_t)$ can be constructed on the same probability space
 starting from the same initial configuration in such a way that $X_t = \bar X_t$ and $Y_t = \bar Y_t$ until the hitting time $T$,
 which we assume from now on.
 Let $T_0 = 0$ and, for all $i \geq 1$, let $T_i$ denote the time of the $i$th jump of the process $\xi_t := \bar X_t + \bar Y_t$.
 Since migration events do not change the value of $\xi_t$, time $T_i$ corresponds to the time of an extinction event at
 $X$ or an expansion event at $Y$, therefore we have
 $$ \begin{array}{l}
     T \ = \ T^+ \quad \hbox{if and only if} \ \hbox{there exists $i \geq 0$} \vspace{2pt} \\
               \hspace{80pt} \hbox{such that $T \in (T_i, T_{i + 1})$ and $\xi_{T_i} > 2 \theta$}. \vspace{4pt} \\
     T \ = \ T^- \quad \hbox{if and only if} \ \hbox{there exists $i \geq 0$} \vspace{2pt} \\
               \hspace{80pt} \hbox{such that $T \in (T_i, T_{i + 1})$ and $\xi_{T_i} < 2 \theta$}. \end{array} $$
 Let $\epsilon > 0$ small such that $1 - \epsilon > 2 \theta$, and consider the events
 $$ \begin{array}{rcl}
  D_{i, n}^- \ & = & \ \{\bar X_{T_i} = 0 \ \hbox{and} \ \bar Y_{T_i} \in 2 \theta - [n \epsilon, (n + 1) \epsilon) \} \vspace{4pt} \\
  D_{i, n}^+ \ & = & \ \{\bar Y_{T_i} = 1 \ \hbox{and} \ \bar X_{T_i} \in [n \epsilon, (n + 1) \epsilon) \}. \end{array} $$
 First, since $1 - (n + 1) \epsilon > 2 \theta - n \epsilon$, migration events between $T_i$ and $T_{i + 1}$ displace
 less individuals on the event $D_{i, n}^-$ than on $D_{i, n}^+$ so
 $$ P \,(T \in (T_i, T_{i + 1}) \ | \ D_{i, n}^-) \ \leq \ P \,(T \in (T_i, T_{i + 1}) \ | \ D_{i, n}^+). $$
 Second, note that $v_n = (1 - \mu)^n < 2 \theta$ if and only if we have
 $$ n \ > \ m_0 \ := \ \lfloor \ln (2 \theta) / \ln (1 - \mu) \rfloor. $$
 In particular, if $T_i$ is the time of an extinction event at $X$ then $\bar Y_{T_i} < 2 \theta$ only if at least $m_0$ migration
 events have occurred since the last expansion event at patch $Y$.
 This implies that
 $$ P \,(\bar Y_{T_i} < 2 \theta) \ \leq \ \bigg(\frac{1}{1 + r} \bigg)^{m_0}. $$
 Since by symmetry the random variables $\bar X_t$ and $1 - \bar Y_t$ are identically distributed, and $2 \theta - \bar X_t$ is
 stochastically smaller than $\bar Y_t$, we deduce that
 $$ P \,(D_{i, n}^-) \ \leq \ P \,(\bar Y_{T_i} < 2 \theta) \ P \,(D_{i, n}^+) \ \leq \
      \bigg(\frac{1}{1 + r} \bigg)^{m_0} P \,(D_{i, n}^+). $$
 Finally, observing that
 $$ \begin{array}{l}
    \{T = T^- \} \ = \ \displaystyle \bigcup_{i = 0}^{\infty} \
    \bigcup_{n = 0}^{\lfloor \epsilon^{-1} \rfloor} \ \{T \in (T_i, T_{i + 1}) \} \,\cap \,D_{n, i}^- \\ \hspace{50pt} \hbox{and} \quad
    \{T = T^+ \} \ \supset \ \displaystyle \bigcup_{i = 0}^{\infty} \
    \bigcup_{n = 0}^{\lfloor \epsilon^{-1} \rfloor} \ \{T \in (T_i, T_{i + 1}) \} \,\cap \,D_{n, i}^+ \end{array} $$
 we can conclude that
 $$ \begin{array}{l}
  P \,(T = T^-) \ = \
    \displaystyle \sum_{i = 0}^{\infty} \ \sum_{n = 0}^{\lfloor \epsilon^{-1} \rfloor} \
  P \,(T \in (T_i, T_{i + 1}) \ | \ D_{n, i}^-) \ P \,(D_{n, i}^-) \\ \hspace{10pt} \leq \
    \displaystyle \sum_{i = 0}^{\infty} \ \sum_{n = 0}^{\lfloor \epsilon^{-1} \rfloor} \
  P \,(T \in (T_i, T_{i + 1}) \ | \ D_{n, i}^+) \ P \,(D_{n, i}^+) \ (P \,(D_{n, i}^-) / P \,(D_{n, i}^+)) \\ \hspace{10pt} \leq \
    \displaystyle \bigg(\frac{1}{1 + r} \bigg)^{m_0} \
    \sum_{i = 0}^{\infty} \ \sum_{n = 0}^{\lfloor \epsilon^{-1} \rfloor} \ P \,(T \in (T_i, T_{i + 1}) \,; D_{n, i}^+) \vspace{6pt} \\ \hspace{10pt} \leq \
    \displaystyle \bigg(\frac{1}{1 + r} \bigg)^{m_0} P \,(T = T^+). \end{array} $$
 This completes the proof. \hspace{2mm} $\blacksquare$



\end{document}

%% file: ODE-1.pstex_t
\begin{picture}(0,0)%
\includegraphics{ODE-1.pstex}%
\end{picture}%
\setlength{\unitlength}{3947sp}%
\begingroup\makeatletter\ifx\SetFigFontNFSS\undefined%
\gdef\SetFigFontNFSS#1#2#3#4#5{%
  \reset@font\fontsize{#1}{#2pt}%
  \fontfamily{#3}\fontseries{#4}\fontshape{#5}%
  \selectfont}%
\fi\endgroup%
\begin{picture}(13783,22086)(-483,-20632)
\put(3601,-20536){\makebox(0,0)[b]{\smash{{\SetFigFontNFSS{20}{24.0}{\familydefault}{\mddefault}{\updefault}$2 \theta$}}}}
\put(  1,-20536){\makebox(0,0)[b]{\smash{{\SetFigFontNFSS{20}{24.0}{\familydefault}{\mddefault}{\updefault}0}}}}
\put(6001,-20536){\makebox(0,0)[b]{\smash{{\SetFigFontNFSS{20}{24.0}{\familydefault}{\mddefault}{\updefault}1}}}}
\put(-299,-16636){\makebox(0,0)[rb]{\smash{{\SetFigFontNFSS{20}{24.0}{\familydefault}{\mddefault}{\updefault}$2 \theta$}}}}
\put(-299,-20236){\makebox(0,0)[rb]{\smash{{\SetFigFontNFSS{20}{24.0}{\familydefault}{\mddefault}{\updefault}0}}}}
\put(-299,-14236){\makebox(0,0)[rb]{\smash{{\SetFigFontNFSS{20}{24.0}{\familydefault}{\mddefault}{\updefault}1}}}}
\put(-299,-15361){\makebox(0,0)[rb]{\smash{{\SetFigFontNFSS{20}{24.0}{\familydefault}{\mddefault}{\updefault}$y$}}}}
\put(3601,-13036){\makebox(0,0)[b]{\smash{{\SetFigFontNFSS{20}{24.0}{\familydefault}{\mddefault}{\updefault}$2 \theta$}}}}
\put(  1,-13036){\makebox(0,0)[b]{\smash{{\SetFigFontNFSS{20}{24.0}{\familydefault}{\mddefault}{\updefault}0}}}}
\put(6001,-13036){\makebox(0,0)[b]{\smash{{\SetFigFontNFSS{20}{24.0}{\familydefault}{\mddefault}{\updefault}1}}}}
\put(-299,-9136){\makebox(0,0)[rb]{\smash{{\SetFigFontNFSS{20}{24.0}{\familydefault}{\mddefault}{\updefault}$2 \theta$}}}}
\put(-299,-12736){\makebox(0,0)[rb]{\smash{{\SetFigFontNFSS{20}{24.0}{\familydefault}{\mddefault}{\updefault}0}}}}
\put(-299,-6736){\makebox(0,0)[rb]{\smash{{\SetFigFontNFSS{20}{24.0}{\familydefault}{\mddefault}{\updefault}1}}}}
\put(-299,-7861){\makebox(0,0)[rb]{\smash{{\SetFigFontNFSS{20}{24.0}{\familydefault}{\mddefault}{\updefault}$y$}}}}
\put(7201,-13036){\makebox(0,0)[b]{\smash{{\SetFigFontNFSS{20}{24.0}{\familydefault}{\mddefault}{\updefault}0}}}}
\put(13201,-13036){\makebox(0,0)[b]{\smash{{\SetFigFontNFSS{20}{24.0}{\familydefault}{\mddefault}{\updefault}1}}}}
\put(6901,-12736){\makebox(0,0)[rb]{\smash{{\SetFigFontNFSS{20}{24.0}{\familydefault}{\mddefault}{\updefault}0}}}}
\put(6901,-6736){\makebox(0,0)[rb]{\smash{{\SetFigFontNFSS{20}{24.0}{\familydefault}{\mddefault}{\updefault}1}}}}
\put(7201,-20536){\makebox(0,0)[b]{\smash{{\SetFigFontNFSS{20}{24.0}{\familydefault}{\mddefault}{\updefault}0}}}}
\put(13201,-20536){\makebox(0,0)[b]{\smash{{\SetFigFontNFSS{20}{24.0}{\familydefault}{\mddefault}{\updefault}1}}}}
\put(6901,-20236){\makebox(0,0)[rb]{\smash{{\SetFigFontNFSS{20}{24.0}{\familydefault}{\mddefault}{\updefault}0}}}}
\put(6901,-14236){\makebox(0,0)[rb]{\smash{{\SetFigFontNFSS{20}{24.0}{\familydefault}{\mddefault}{\updefault}1}}}}
\put(3601,-5536){\makebox(0,0)[b]{\smash{{\SetFigFontNFSS{20}{24.0}{\familydefault}{\mddefault}{\updefault}$2 \theta$}}}}
\put(  1,-5536){\makebox(0,0)[b]{\smash{{\SetFigFontNFSS{20}{24.0}{\familydefault}{\mddefault}{\updefault}0}}}}
\put(6001,-5536){\makebox(0,0)[b]{\smash{{\SetFigFontNFSS{20}{24.0}{\familydefault}{\mddefault}{\updefault}1}}}}
\put(-299,-1636){\makebox(0,0)[rb]{\smash{{\SetFigFontNFSS{20}{24.0}{\familydefault}{\mddefault}{\updefault}$2 \theta$}}}}
\put(-299,-5236){\makebox(0,0)[rb]{\smash{{\SetFigFontNFSS{20}{24.0}{\familydefault}{\mddefault}{\updefault}0}}}}
\put(-299,764){\makebox(0,0)[rb]{\smash{{\SetFigFontNFSS{20}{24.0}{\familydefault}{\mddefault}{\updefault}1}}}}
\put(-299,-361){\makebox(0,0)[rb]{\smash{{\SetFigFontNFSS{20}{24.0}{\familydefault}{\mddefault}{\updefault}$y$}}}}
\put(7201,-5536){\makebox(0,0)[b]{\smash{{\SetFigFontNFSS{20}{24.0}{\familydefault}{\mddefault}{\updefault}0}}}}
\put(13201,-5536){\makebox(0,0)[b]{\smash{{\SetFigFontNFSS{20}{24.0}{\familydefault}{\mddefault}{\updefault}1}}}}
\put(6901,-5236){\makebox(0,0)[rb]{\smash{{\SetFigFontNFSS{20}{24.0}{\familydefault}{\mddefault}{\updefault}0}}}}
\put(6901,764){\makebox(0,0)[rb]{\smash{{\SetFigFontNFSS{20}{24.0}{\familydefault}{\mddefault}{\updefault}1}}}}
\put(6901,-361){\makebox(0,0)[rb]{\smash{{\SetFigFontNFSS{20}{24.0}{\familydefault}{\mddefault}{\updefault}$2 \theta$}}}}
\put(6901,-1636){\makebox(0,0)[rb]{\smash{{\SetFigFontNFSS{20}{24.0}{\familydefault}{\mddefault}{\updefault}$y$}}}}
\put(6901,-7861){\makebox(0,0)[rb]{\smash{{\SetFigFontNFSS{20}{24.0}{\familydefault}{\mddefault}{\updefault}$2 \theta$}}}}
\put(6901,-9136){\makebox(0,0)[rb]{\smash{{\SetFigFontNFSS{20}{24.0}{\familydefault}{\mddefault}{\updefault}$y$}}}}
\put(6901,-15361){\makebox(0,0)[rb]{\smash{{\SetFigFontNFSS{20}{24.0}{\familydefault}{\mddefault}{\updefault}$2 \theta$}}}}
\put(6901,-16636){\makebox(0,0)[rb]{\smash{{\SetFigFontNFSS{20}{24.0}{\familydefault}{\mddefault}{\updefault}$y$}}}}
\put(4801,-5536){\makebox(0,0)[b]{\smash{{\SetFigFontNFSS{20}{24.0}{\familydefault}{\mddefault}{\updefault}$x$}}}}
\put(4801,-13036){\makebox(0,0)[b]{\smash{{\SetFigFontNFSS{20}{24.0}{\familydefault}{\mddefault}{\updefault}$x$}}}}
\put(4801,-20536){\makebox(0,0)[b]{\smash{{\SetFigFontNFSS{20}{24.0}{\familydefault}{\mddefault}{\updefault}$x$}}}}
\put(10801,-5536){\makebox(0,0)[b]{\smash{{\SetFigFontNFSS{20}{24.0}{\familydefault}{\mddefault}{\updefault}$x$}}}}
\put(12001,-5536){\makebox(0,0)[b]{\smash{{\SetFigFontNFSS{20}{24.0}{\familydefault}{\mddefault}{\updefault}$2 \theta$}}}}
\put(10801,-13036){\makebox(0,0)[b]{\smash{{\SetFigFontNFSS{20}{24.0}{\familydefault}{\mddefault}{\updefault}$x$}}}}
\put(12001,-13036){\makebox(0,0)[b]{\smash{{\SetFigFontNFSS{20}{24.0}{\familydefault}{\mddefault}{\updefault}$2 \theta$}}}}
\put(10801,-20536){\makebox(0,0)[b]{\smash{{\SetFigFontNFSS{20}{24.0}{\familydefault}{\mddefault}{\updefault}$x$}}}}
\put(12001,-20536){\makebox(0,0)[b]{\smash{{\SetFigFontNFSS{20}{24.0}{\familydefault}{\mddefault}{\updefault}$2 \theta$}}}}
\put(3001,1139){\makebox(0,0)[b]{\smash{{\SetFigFontNFSS{20}{24.0}{\familydefault}{\mddefault}{\updefault}(a) \ $\theta = 0.30$ and $\mu = 0.10$}}}}
\put(10201,1139){\makebox(0,0)[b]{\smash{{\SetFigFontNFSS{20}{24.0}{\familydefault}{\mddefault}{\updefault}(b) \ $\theta = 0.40$ and $\mu = 0.10$}}}}
\put(3001,-6361){\makebox(0,0)[b]{\smash{{\SetFigFontNFSS{20}{24.0}{\familydefault}{\mddefault}{\updefault}(c) \ $\theta = 0.30$ and $\mu = 0.04$}}}}
\put(10201,-6361){\makebox(0,0)[b]{\smash{{\SetFigFontNFSS{20}{24.0}{\familydefault}{\mddefault}{\updefault}(d) \ $\theta = 0.40$ and $\mu = 0.04$}}}}
\put(3001,-13861){\makebox(0,0)[b]{\smash{{\SetFigFontNFSS{20}{24.0}{\familydefault}{\mddefault}{\updefault}(e) \ $\theta = 0.30$ and $\mu = 0.02$}}}}
\put(10201,-13861){\makebox(0,0)[b]{\smash{{\SetFigFontNFSS{20}{24.0}{\familydefault}{\mddefault}{\updefault}(f) \ $\theta = 0.40$ and $\mu = 0.02$}}}}
\end{picture}%

%% file: ODE-2.pstex_t
\begin{picture}(0,0)%
\includegraphics{ODE-2.pstex}%
\end{picture}%
\setlength{\unitlength}{3947sp}%
\begingroup\makeatletter\ifx\SetFigFontNFSS\undefined%
\gdef\SetFigFontNFSS#1#2#3#4#5{%
  \reset@font\fontsize{#1}{#2pt}%
  \fontfamily{#3}\fontseries{#4}\fontshape{#5}%
  \selectfont}%
\fi\endgroup%
\begin{picture}(13783,7086)(-483,-5632)
\put(7201,-5536){\makebox(0,0)[b]{\smash{{\SetFigFontNFSS{20}{24.0}{\familydefault}{\mddefault}{\updefault}0}}}}
\put(13201,-5536){\makebox(0,0)[b]{\smash{{\SetFigFontNFSS{20}{24.0}{\familydefault}{\mddefault}{\updefault}1}}}}
\put(6901,-5236){\makebox(0,0)[rb]{\smash{{\SetFigFontNFSS{20}{24.0}{\familydefault}{\mddefault}{\updefault}0}}}}
\put(6901,764){\makebox(0,0)[rb]{\smash{{\SetFigFontNFSS{20}{24.0}{\familydefault}{\mddefault}{\updefault}1}}}}
\put(6901,-1636){\makebox(0,0)[rb]{\smash{{\SetFigFontNFSS{20}{24.0}{\familydefault}{\mddefault}{\updefault}$y$}}}}
\put(10801,-5536){\makebox(0,0)[b]{\smash{{\SetFigFontNFSS{20}{24.0}{\familydefault}{\mddefault}{\updefault}$x$}}}}
\put(  1,-5536){\makebox(0,0)[b]{\smash{{\SetFigFontNFSS{20}{24.0}{\familydefault}{\mddefault}{\updefault}0}}}}
\put(6001,-5536){\makebox(0,0)[b]{\smash{{\SetFigFontNFSS{20}{24.0}{\familydefault}{\mddefault}{\updefault}1}}}}
\put(-299,-5236){\makebox(0,0)[rb]{\smash{{\SetFigFontNFSS{20}{24.0}{\familydefault}{\mddefault}{\updefault}0}}}}
\put(-299,764){\makebox(0,0)[rb]{\smash{{\SetFigFontNFSS{20}{24.0}{\familydefault}{\mddefault}{\updefault}1}}}}
\put(-299,-361){\makebox(0,0)[rb]{\smash{{\SetFigFontNFSS{20}{24.0}{\familydefault}{\mddefault}{\updefault}$y$}}}}
\put(4801,-5536){\makebox(0,0)[b]{\smash{{\SetFigFontNFSS{20}{24.0}{\familydefault}{\mddefault}{\updefault}$x$}}}}
\put(3001,1139){\makebox(0,0)[b]{\smash{{\SetFigFontNFSS{20}{24.0}{\familydefault}{\mddefault}{\updefault}(a) \ $\theta = 0.70$ and $\mu = 0.04$}}}}
\put(10201,1139){\makebox(0,0)[b]{\smash{{\SetFigFontNFSS{20}{24.0}{\familydefault}{\mddefault}{\updefault}(b) \ $\theta = 0.60$ and $\mu = 0.04$}}}}
\end{picture}%

%% file: dynamics.pstex_t
\begin{picture}(0,0)%
\includegraphics{dynamics.pstex}%
\end{picture}%
\setlength{\unitlength}{3947sp}%
\begingroup\makeatletter\ifx\SetFigFontNFSS\undefined%
\gdef\SetFigFontNFSS#1#2#3#4#5{%
  \reset@font\fontsize{#1}{#2pt}%
  \fontfamily{#3}\fontseries{#4}\fontshape{#5}%
  \selectfont}%
\fi\endgroup%
\begin{picture}(11730,5991)(-314,-4525)
\put(2401,1139){\makebox(0,0)[b]{\smash{{\SetFigFontNFSS{20}{24.0}{\familydefault}{\mddefault}{\updefault}{\color[rgb]{0,0,0}patch $X$}%
}}}}
\put(8701,1139){\makebox(0,0)[b]{\smash{{\SetFigFontNFSS{20}{24.0}{\familydefault}{\mddefault}{\updefault}{\color[rgb]{0,0,0}patch $Y$}%
}}}}
\put(5551,-3211){\makebox(0,0)[b]{\smash{{\SetFigFontNFSS{20}{24.0}{\familydefault}{\mddefault}{\updefault}{\color[rgb]{0,0,0}rate 1}%
}}}}
\put(  1,-4411){\makebox(0,0)[b]{\smash{{\SetFigFontNFSS{20}{24.0}{\familydefault}{\mddefault}{\updefault}{\color[rgb]{0,0,0}0}%
}}}}
\put(3301,-4411){\makebox(0,0)[b]{\smash{{\SetFigFontNFSS{20}{24.0}{\familydefault}{\mddefault}{\updefault}{\color[rgb]{0,0,0}$1 - \mu$}%
}}}}
\put(4801,-4411){\makebox(0,0)[b]{\smash{{\SetFigFontNFSS{20}{24.0}{\familydefault}{\mddefault}{\updefault}{\color[rgb]{0,0,0}1}%
}}}}
\put(6301,-4411){\makebox(0,0)[b]{\smash{{\SetFigFontNFSS{20}{24.0}{\familydefault}{\mddefault}{\updefault}{\color[rgb]{0,0,0}0}%
}}}}
\put(7801,-4411){\makebox(0,0)[b]{\smash{{\SetFigFontNFSS{20}{24.0}{\familydefault}{\mddefault}{\updefault}{\color[rgb]{0,0,0}$\mu$}%
}}}}
\put(11101,-4411){\makebox(0,0)[b]{\smash{{\SetFigFontNFSS{20}{24.0}{\familydefault}{\mddefault}{\updefault}{\color[rgb]{0,0,0}1}%
}}}}
\put(-299,-2911){\makebox(0,0)[rb]{\smash{{\SetFigFontNFSS{20}{24.0}{\familydefault}{\mddefault}{\updefault}{\color[rgb]{0,0,0}$X_t$}%
}}}}
\put(-299,-2011){\makebox(0,0)[rb]{\smash{{\SetFigFontNFSS{20}{24.0}{\familydefault}{\mddefault}{\updefault}{\color[rgb]{0,0,0}$\theta$}%
}}}}
\put(11401,-811){\makebox(0,0)[lb]{\smash{{\SetFigFontNFSS{20}{24.0}{\familydefault}{\mddefault}{\updefault}{\color[rgb]{0,0,0}$Y_t$}%
}}}}
\put(11401,-2011){\makebox(0,0)[lb]{\smash{{\SetFigFontNFSS{20}{24.0}{\familydefault}{\mddefault}{\updefault}{\color[rgb]{0,0,0}$\theta$}%
}}}}
\put(9001,-61){\makebox(0,0)[lb]{\smash{{\SetFigFontNFSS{20}{24.0}{\familydefault}{\mddefault}{\updefault}{\color[rgb]{0,0,0}rate $r$}%
}}}}
\put(2101,-3511){\makebox(0,0)[rb]{\smash{{\SetFigFontNFSS{20}{24.0}{\familydefault}{\mddefault}{\updefault}{\color[rgb]{0,0,0}rate $r$}%
}}}}
\end{picture}%

%% file: proba-time.pstex_t
\begin{picture}(0,0)%
\includegraphics{proba-time.pstex}%
\end{picture}%
\setlength{\unitlength}{3947sp}%
\begingroup\makeatletter\ifx\SetFigFontNFSS\undefined%
\gdef\SetFigFontNFSS#1#2#3#4#5{%
  \reset@font\fontsize{#1}{#2pt}%
  \fontfamily{#3}\fontseries{#4}\fontshape{#5}%
  \selectfont}%
\fi\endgroup%
\begin{picture}(10800,6831)(-715,-5623)
\put(1201,-5536){\makebox(0,0)[b]{\smash{{\SetFigFontNFSS{14}{16.8}{\familydefault}{\mddefault}{\updefault}{\color[rgb]{0,0,0}0.1}%
}}}}
\put(2401,-5536){\makebox(0,0)[b]{\smash{{\SetFigFontNFSS{14}{16.8}{\familydefault}{\mddefault}{\updefault}{\color[rgb]{0,0,0}0.2}%
}}}}
\put(3601,-5536){\makebox(0,0)[b]{\smash{{\SetFigFontNFSS{14}{16.8}{\familydefault}{\mddefault}{\updefault}{\color[rgb]{0,0,0}0.3}%
}}}}
\put(6001,-5536){\makebox(0,0)[b]{\smash{{\SetFigFontNFSS{14}{16.8}{\familydefault}{\mddefault}{\updefault}{\color[rgb]{0,0,0}0.5}%
}}}}
\put(4801,-5536){\makebox(0,0)[b]{\smash{{\SetFigFontNFSS{14}{16.8}{\familydefault}{\mddefault}{\updefault}{\color[rgb]{0,0,0}$\mu$}%
}}}}
\put(  1,-5536){\makebox(0,0)[b]{\smash{{\SetFigFontNFSS{14}{16.8}{\familydefault}{\mddefault}{\updefault}{\color[rgb]{0,0,0}0.02}%
}}}}
\put(3001,989){\makebox(0,0)[b]{\smash{{\SetFigFontNFSS{14}{16.8}{\familydefault}{\mddefault}{\updefault}{\color[rgb]{0,0,0}probability of extinction}%
}}}}
\put(-224,-2311){\makebox(0,0)[rb]{\smash{{\SetFigFontNFSS{14}{16.8}{\familydefault}{\mddefault}{\updefault}{\color[rgb]{0,0,0}0.50}%
}}}}
\put(-224,689){\makebox(0,0)[rb]{\smash{{\SetFigFontNFSS{14}{16.8}{\familydefault}{\mddefault}{\updefault}{\color[rgb]{0,0,0}0.75}%
}}}}
\put(-224,-661){\makebox(0,0)[rb]{\smash{{\SetFigFontNFSS{14}{16.8}{\familydefault}{\mddefault}{\updefault}{\color[rgb]{0,0,0}$\theta$}%
}}}}
\put(-224,-5236){\makebox(0,0)[rb]{\smash{{\SetFigFontNFSS{14}{16.8}{\familydefault}{\mddefault}{\updefault}{\color[rgb]{0,0,0}0.25}%
}}}}
\put(7276,-2311){\makebox(0,0)[rb]{\smash{{\SetFigFontNFSS{14}{16.8}{\familydefault}{\mddefault}{\updefault}{\color[rgb]{0,0,0}0.50}%
}}}}
\put(7276,689){\makebox(0,0)[rb]{\smash{{\SetFigFontNFSS{14}{16.8}{\familydefault}{\mddefault}{\updefault}{\color[rgb]{0,0,0}0.75}%
}}}}
\put(7276,-661){\makebox(0,0)[rb]{\smash{{\SetFigFontNFSS{14}{16.8}{\familydefault}{\mddefault}{\updefault}{\color[rgb]{0,0,0}$\theta$}%
}}}}
\put(7276,-5236){\makebox(0,0)[rb]{\smash{{\SetFigFontNFSS{14}{16.8}{\familydefault}{\mddefault}{\updefault}{\color[rgb]{0,0,0}0.25}%
}}}}
\put(9901,-5536){\makebox(0,0)[b]{\smash{{\SetFigFontNFSS{14}{16.8}{\familydefault}{\mddefault}{\updefault}{\color[rgb]{0,0,0}0.2}%
}}}}
\put(8701,-5536){\makebox(0,0)[b]{\smash{{\SetFigFontNFSS{14}{16.8}{\familydefault}{\mddefault}{\updefault}{\color[rgb]{0,0,0}$\mu$}%
}}}}
\put(7501,-5536){\makebox(0,0)[b]{\smash{{\SetFigFontNFSS{14}{16.8}{\familydefault}{\mddefault}{\updefault}{\color[rgb]{0,0,0}0.02}%
}}}}
\put(8701,989){\makebox(0,0)[b]{\smash{{\SetFigFontNFSS{14}{16.8}{\familydefault}{\mddefault}{\updefault}{\color[rgb]{0,0,0}time to fixation}%
}}}}
\put(5401,-4561){\makebox(0,0)[rb]{\smash{{\SetFigFontNFSS{29}{34.8}{\familydefault}{\mddefault}{\updefault}{\color[rgb]{0,0,0}2}%
}}}}
\put(5401,-136){\makebox(0,0)[rb]{\smash{{\SetFigFontNFSS{29}{34.8}{\familydefault}{\mddefault}{\updefault}{\color[rgb]{1,1,1}1}%
}}}}
\end{picture}%

%% file: nullclines.pstex_t
\begin{picture}(0,0)%
\includegraphics{nullclines.pstex}%
\end{picture}%
\setlength{\unitlength}{3947sp}%
\begingroup\makeatletter\ifx\SetFigFontNFSS\undefined%
\gdef\SetFigFontNFSS#1#2#3#4#5{%
  \reset@font\fontsize{#1}{#2pt}%
  \fontfamily{#3}\fontseries{#4}\fontshape{#5}%
  \selectfont}%
\fi\endgroup%
\begin{picture}(17687,6141)(-571,-4675)
\put(-449,-511){\makebox(0,0)[b]{\smash{{\SetFigFontNFSS{20}{24.0}{\familydefault}{\mddefault}{\updefault}{\color[rgb]{0,0,0}1}%
}}}}
\put(3601,-4561){\makebox(0,0)[b]{\smash{{\SetFigFontNFSS{20}{24.0}{\familydefault}{\mddefault}{\updefault}{\color[rgb]{0,0,0}1}%
}}}}
\put(5101,-4111){\makebox(0,0)[b]{\smash{{\SetFigFontNFSS{20}{24.0}{\familydefault}{\mddefault}{\updefault}{\color[rgb]{0,0,0}$x$}%
}}}}
\put(  1,1139){\makebox(0,0)[b]{\smash{{\SetFigFontNFSS{20}{24.0}{\familydefault}{\mddefault}{\updefault}{\color[rgb]{0,0,0}$y$}%
}}}}
\put(5551,-511){\makebox(0,0)[b]{\smash{{\SetFigFontNFSS{20}{24.0}{\familydefault}{\mddefault}{\updefault}{\color[rgb]{0,0,0}1}%
}}}}
\put(9601,-4561){\makebox(0,0)[b]{\smash{{\SetFigFontNFSS{20}{24.0}{\familydefault}{\mddefault}{\updefault}{\color[rgb]{0,0,0}1}%
}}}}
\put(11101,-4111){\makebox(0,0)[b]{\smash{{\SetFigFontNFSS{20}{24.0}{\familydefault}{\mddefault}{\updefault}{\color[rgb]{0,0,0}$x$}%
}}}}
\put(6001,1139){\makebox(0,0)[b]{\smash{{\SetFigFontNFSS{20}{24.0}{\familydefault}{\mddefault}{\updefault}{\color[rgb]{0,0,0}$y$}%
}}}}
\put(11551,-511){\makebox(0,0)[b]{\smash{{\SetFigFontNFSS{20}{24.0}{\familydefault}{\mddefault}{\updefault}{\color[rgb]{0,0,0}1}%
}}}}
\put(15601,-4561){\makebox(0,0)[b]{\smash{{\SetFigFontNFSS{20}{24.0}{\familydefault}{\mddefault}{\updefault}{\color[rgb]{0,0,0}1}%
}}}}
\put(17101,-4111){\makebox(0,0)[b]{\smash{{\SetFigFontNFSS{20}{24.0}{\familydefault}{\mddefault}{\updefault}{\color[rgb]{0,0,0}$x$}%
}}}}
\put(12001,1139){\makebox(0,0)[b]{\smash{{\SetFigFontNFSS{20}{24.0}{\familydefault}{\mddefault}{\updefault}{\color[rgb]{0,0,0}$y$}%
}}}}
\put(7201,-4561){\makebox(0,0)[b]{\smash{{\SetFigFontNFSS{20}{24.0}{\familydefault}{\mddefault}{\updefault}{\color[rgb]{0,0,0}$\theta$}%
}}}}
\put(5551,-2911){\makebox(0,0)[b]{\smash{{\SetFigFontNFSS{20}{24.0}{\familydefault}{\mddefault}{\updefault}{\color[rgb]{0,0,0}$\theta$}%
}}}}
\put(13201,-4561){\makebox(0,0)[b]{\smash{{\SetFigFontNFSS{20}{24.0}{\familydefault}{\mddefault}{\updefault}{\color[rgb]{0,0,0}$\theta$}%
}}}}
\put(11551,-2911){\makebox(0,0)[b]{\smash{{\SetFigFontNFSS{20}{24.0}{\familydefault}{\mddefault}{\updefault}{\color[rgb]{0,0,0}$\theta$}%
}}}}
\put(-449,-2911){\makebox(0,0)[b]{\smash{{\SetFigFontNFSS{20}{24.0}{\familydefault}{\mddefault}{\updefault}{\color[rgb]{0,0,0}$\theta$}%
}}}}
\put(1201,-4561){\makebox(0,0)[b]{\smash{{\SetFigFontNFSS{20}{24.0}{\familydefault}{\mddefault}{\updefault}{\color[rgb]{0,0,0}$\theta$}%
}}}}
\end{picture}%

%% file: grid.pstex_t
\begin{picture}(0,0)%
\includegraphics{grid.pstex}%
\end{picture}%
\setlength{\unitlength}{3947sp}%
\begingroup\makeatletter\ifx\SetFigFontNFSS\undefined%
\gdef\SetFigFontNFSS#1#2#3#4#5{%
  \reset@font\fontsize{#1}{#2pt}%
  \fontfamily{#3}\fontseries{#4}\fontshape{#5}%
  \selectfont}%
\fi\endgroup%
\begin{picture}(11598,7492)(-314,-6113)
\put(-224,1064){\makebox(0,0)[b]{\smash{{\SetFigFontNFSS{20}{24.0}{\familydefault}{\mddefault}{\updefault}{\color[rgb]{0,0,0}$(0, 1)$}%
}}}}
\put(-299,-211){\makebox(0,0)[rb]{\smash{{\SetFigFontNFSS{20}{24.0}{\familydefault}{\mddefault}{\updefault}{\color[rgb]{0,0,0}$v_1$}%
}}}}
\put(-299,-2686){\makebox(0,0)[rb]{\smash{{\SetFigFontNFSS{20}{24.0}{\familydefault}{\mddefault}{\updefault}{\color[rgb]{0,0,0}$v_j$}%
}}}}
\put(-299,-3436){\makebox(0,0)[rb]{\smash{{\SetFigFontNFSS{20}{24.0}{\familydefault}{\mddefault}{\updefault}{\color[rgb]{0,0,0}$v_{j + 1}$}%
}}}}
\put(3301,1064){\makebox(0,0)[b]{\smash{{\SetFigFontNFSS{20}{24.0}{\familydefault}{\mddefault}{\updefault}{\color[rgb]{0,0,0}$u_{i + 1}$}%
}}}}
\put(2551,1064){\makebox(0,0)[b]{\smash{{\SetFigFontNFSS{20}{24.0}{\familydefault}{\mddefault}{\updefault}{\color[rgb]{0,0,0}$u_i$}%
}}}}
\put(6076,1064){\makebox(0,0)[b]{\smash{{\SetFigFontNFSS{20}{24.0}{\familydefault}{\mddefault}{\updefault}{\color[rgb]{0,0,0}$(0, 1)$}%
}}}}
\put(11251,1064){\makebox(0,0)[b]{\smash{{\SetFigFontNFSS{20}{24.0}{\familydefault}{\mddefault}{\updefault}{\color[rgb]{0,0,0}$u_{n_1}$}%
}}}}
\put(9601,1064){\makebox(0,0)[b]{\smash{{\SetFigFontNFSS{20}{24.0}{\familydefault}{\mddefault}{\updefault}{\color[rgb]{0,0,0}$u_{n_0}$}%
}}}}
\put(6001,-2686){\makebox(0,0)[rb]{\smash{{\SetFigFontNFSS{20}{24.0}{\familydefault}{\mddefault}{\updefault}{\color[rgb]{0,0,0}$v_{n_0}$}%
}}}}
\put(8551,-1561){\makebox(0,0)[rb]{\smash{{\SetFigFontNFSS{20}{24.0}{\familydefault}{\mddefault}{\updefault}{\color[rgb]{0,0,0}$N$}%
}}}}
\put(2251,-661){\makebox(0,0)[rb]{\smash{{\SetFigFontNFSS{20}{24.0}{\familydefault}{\mddefault}{\updefault}{\color[rgb]{0,0,0}$N$}%
}}}}
\put(1351,-2311){\makebox(0,0)[b]{\smash{{\SetFigFontNFSS{20}{24.0}{\familydefault}{\mddefault}{\updefault}{\color[rgb]{0,0,0}$W$}%
}}}}
\put(7651,-661){\makebox(0,0)[b]{\smash{{\SetFigFontNFSS{20}{24.0}{\familydefault}{\mddefault}{\updefault}{\color[rgb]{0,0,0}$W$}%
}}}}
\put(2926,-3886){\makebox(0,0)[b]{\smash{{\SetFigFontNFSS{20}{24.0}{\familydefault}{\mddefault}{\updefault}{\color[rgb]{0,0,0}$SE$}%
}}}}
\put(6001,-5986){\makebox(0,0)[rb]{\smash{{\SetFigFontNFSS{20}{24.0}{\familydefault}{\mddefault}{\updefault}{\color[rgb]{0,0,0}$v_{n_2}$}%
}}}}
\end{picture}%